\documentclass{amsart}
\usepackage[utf8]{inputenc}
\usepackage[table]{xcolor}
\usepackage{mathtools,tikz-cd,amsmath,amssymb}
\usepackage{enumerate}
\usepackage[parfill]{parskip}

\usepackage{tikz,tikz-3dplot}
\usetikzlibrary{shapes.misc}
\tikzset{cross/.style={cross out, draw=black, minimum size=2*(#1-\pgflinewidth), inner sep=0pt, outer sep=0pt},
cross/.default={2pt}}

\usepackage{amsthm,amssymb,latexsym,epic,bbm,comment,mathrsfs}
\usepackage[all,2cell]{xy}
\xyoption{2cell}
\allowdisplaybreaks

\newtheorem{theorem}{Theorem}

\newtheorem{lemma}[theorem]{Lemma}
\newtheorem{corollary}[theorem]{Corollary}

\newtheorem{proposition}[theorem]{Proposition}

\theoremstyle{definition}%below are displayed non-italic

\newcommand{\ext}{\operatorname{ext}}
\newcommand{\Ext}{\operatorname{Ext}}

\newcommand{\inv}{^{-1}}

\newcommand{\cO}{\mathcal O}

\newcommand{\soc}{\operatorname{soc}}

\newcommand{\Ess}{\operatorname{Ess}}

\newcommand{\surj}{\rightarrow\mathrel{\mkern-14mu}\rightarrow}

\newcommand{\inj}{\hookrightarrow}

\newcommand{\kl}{\underline H}

\newcommand{\jc}{\mathcal J}

\begin{document}

\title[Bigrassmannian permutations and Verma modules]
{Bigrassmannian permutations and Verma modules}

\author[H.~Ko, V.~Mazorchuk and R.~Mr{\dj}en]
{Hankyung Ko, Volodymyr Mazorchuk and Rafael Mr{\dj}en}

\begin{abstract}
We show that bigrassmannian permutations determine the socle 
of the cokernel of an inclusion of Verma modules in type $A$. 
All such socular constituents turn out to be indexed by 
Weyl group elements from the penultimate two-sided cell.
Combinatorially, the socular constituents in the cokernel of
the inclusion of a Verma module indexed by $w\in S_n$
into the dominant Verma module are shown to be determined
by the essential set of $w$ and their degrees in the graded 
picture are shown to be computable in terms of the 
associated rank function. As an application, we compute the 
first extension from a simple module to a Verma module.
\end{abstract}

\maketitle

%\setcounter{tocdepth}{1}
%\tableofcontents

\section{Introduction, motivation and description of results}\label{s1}

\subsection{Setup}\label{s1.1}

Consider the symmetric group $S_n$ on $\{1,2,\dots,n\}$ as a Coxeter group
with simple reflections $s_{i}$, given by the elementary transpositions
$(i,i+1)$, where $i=1,2,\dots,n-1$. Denote by $\leq$ the corresponding
Bruhat order. 
Given $w\in S_n$, we call $i$ a {\em left descent} of $w$ if $s_iw<w$, and a {\em right descent} of $w$ if $ws_i<w$.
An element $w\in S_n$ is called {\em bigrassmannian} provided
that it has exactly one left descent, and exactly one right descent.
Various aspects of bigrassmannian permutations were studied in 
\cite{LS,Ko1,Ko2,EL,EH,Re,RWY}. The most relevant for this paper is the
property that bigrassmannian permutations are exactly the join-irreducible 
elements of $S_n$ with respect to $\leq$, see \cite{LS}. Join-irreducible elements of Weyl groups appear in representation-theoretic context in \cite{IRRT}.
Bigrassmannian elements are the first protagonists of the present paper.
We denote by $\mathbf{B}_n$ the set of all bigrassmannian permutations in $S_n$.

Denote by $\mathfrak{g}$ the simple Lie algebra $\mathfrak{sl}_n$ over $\mathbb{C}$
with the standard triangular decomposition
\begin{displaymath}
\mathfrak{g}=\mathfrak{n}_-\oplus\mathfrak{h}\oplus\mathfrak{n}_+,
\end{displaymath}
where $\mathfrak{h}$ is the Cartan subalgebra of all (traceless) diagonal matrices
and $\mathfrak{n}_+$ is the nilpotent subalgebra of all strictly upper triangular matrices.
Consider the BGG category $\mathcal{O}$ associated to this triangular
decomposition and its principal block $\mathcal{O}_0$, see \cite{BGG,Hu}.

The group $S_n$ is the Weyl group of $\mathfrak{g}$, and thus acts naturally on
$\mathfrak{h}^*$. We also consider the {\em dot-action} of $S_n$ on 
$\mathfrak{h}^*$, that is
\[w\cdot \lambda = w(\lambda+\rho)-\rho\] 
for $\lambda\in\mathfrak{h}^*$, where $\rho$ is the half of the sum of all positive roots
in $\mathfrak{h}^*$. For $\lambda\in \mathfrak{h}^*$, we denote by 
$\Delta(\lambda)$ the Verma module with highest weight $\lambda$ and by
$L(\lambda)$ the unique simple quotient of $\Delta(\lambda)$, see \cite{Ve,Di,Hu}.
The isomorphism classes of simple objects in $\mathcal{O}_0$ are naturally indexed by
the elements in $S_n$ as follows: $S_n\ni w\mapsto L(w\cdot 0)$, where $0$ denotes
the zero element of $\mathfrak{h}^*$. For $w\in S_n$, we set
$\Delta_w:=\Delta(w\cdot 0)$ and $L_w:=L(w\cdot 0)$.

It is well-known, see \cite[Chapter~7]{Di},
that, for $x,y\in S_n$, we have: 
\begin{itemize}
\item $\dim\mathrm{Hom}_{\mathfrak{g}}(\Delta_x,\Delta_y)\leq 1$,
\item a non-zero homomorphism from $\Delta_x$ to $\Delta_y$ exists
if and only if $x\geq y$,
\item each non-zero homomorphism from $\Delta_x$ to $\Delta_y$
is injective.
\end{itemize}
In particular, all $\Delta_w$, where $w\in S_n$, are uniquely determined
submodules  of the {\em dominant Verma module} $\Delta_e$.

The composition multiplicity of $L_x$ in $\Delta_x$ can be computed in terms
of {\em Kazhdan-Lusztig (KL) polynomials}, by \cite{KL,BB,BK,EW}, see Subsection \ref{s2.3}. The associated
{\em KL combinatorics} divides $S_n$ into subsets, called {\em two-sided cells},
ordered with respect to the {\em two-sided order} $\leq_{\mathtt{J}}$.
This coincides with the division of $S_n$ given by the Robinson-Schensted
correspondence: to each $w\in S_n$, we associate a pair $(P_w,Q_w)$ of
standard Young tableaux of the same shape $\lambda$ which is a partition of $n$,
see \cite{Ro,Sch,Sa}. The two-sided cells in $S_n$ are indexed by such $\lambda$ and
correspond precisely to the fibers of the map from $S_n$ to the set of all partitions
of $n$ induced by the Robinson-Schensted correspondence. Moreover, the 
two-sided order coincides with the dominance order on partitions, see \cite{Ge}.
The longest element $w_0$ of $S_n$ forms a two-sided cell which is the maximum with respect to the two-sided order. If we delete this two-sided cell,
in what is left there is again a unique maximum two-sided cell, which we denote by
$\mathcal{J}$ and call the {\em penultimate} two-sided cell. 
The two-sided cell $\mathcal J$ is the second protagonist of the present paper.
Under the Robinson-Schensted correspondence, it corresponds to the partition
$(2,1^{n-2})$ of $n$. Under the involution $x\mapsto xw_0$, the two-sided cell 
$\mathcal{J}$ corresponds to the two-sided cell which contains all simple reflections, studied in, for example, \cite{KMMZ,KM}.  The Kazhdan-Lusztig cell representation of $S_n$ associated
with (any left cell inside) 
$\mathcal{J}$ is exactly the representation of $S_n$ on $\mathfrak{h}^*$.
For each $i,j\in\{1,2,\dots,n-1\}$, the cell $\mathcal{J}$
contains a unique element $w$ such that $s_iw>w$ and $ws_j>w$. We denote
this $w$ by $w_{i,j}$.

\subsection{Motivation}\label{s1.2}

The original motivation for this paper comes from a question by Sascha Orlik and Matthias Strauch
which is discussed in more detail in 
the next subsection. A very special case of that question
leads to the problem of determining $\mathrm{Ext}_{\cO}^1(L_x,\Delta_y)$
in the case $x\neq w_0$. In the case $x=w_0$, we note that $L_{w_0}=\Delta_{w_0}$
and hence the computation of $\mathrm{Ext}_{\cO}^1(L_x,\Delta_y)$
can be reduced, using twisting functors, to \cite[Theorem~32]{Ma} (see the proof of Corollary \ref{cor2}).
The case $x\neq w_0$ requires new techniques and is completed
in Corollary~\ref{cor2}.

\subsection{Description of the main results}\label{s1.3}

The main result of this paper is the following theorem
which, in particular, reveals a completely unexpected connection
between $\mathbf{B}_n$ and $\mathcal{J}$. 

\begin{theorem}\label{thm1}
{\hspace{1mm}}

\begin{enumerate}[$($i$)$]
\item\label{thm1.1} For $w\in S_n$, the module $\Delta_e/\Delta_w$
has simple socle if and only if $w\in \mathbf{B}_n$.
\item\label{thm1.2} The map $\mathbf{B}_n\ni w\mapsto \mathrm{soc}(\Delta_e/\Delta_w)$
induces a bijection between $\mathbf{B}_n$ and simple subquotients of
$\Delta_e$ of the form $L_x$, where $x\in\mathcal{J}$.
\item\label{thm1.3} For $w\in S_n$, the simple subquotients of 
$\Delta_e/\Delta_w$ of the form $L_x$, where $x\in\mathcal{J}$, correspond under the
bijection from \eqref{thm1.2}, to $y\in \mathbf{B}_n$ such that $y\leq w$.
\item\label{thm1.4} For $w\in S_n$, the socle of $\Delta_e/\Delta_w$ 
consists of all $L_x$, where $x\in\mathcal{J}$, which correspond, under the
bijection from \eqref{thm1.2}, to the Bruhat maximal elements in
$\{y\in \mathbf{B}_n\,:\, y\leq w\}$. 
\end{enumerate}
\end{theorem}

The bijection from Theorem~\ref{thm1}(\ref{thm1.2}) is explicitly given in Subsection \ref{subs_explicit}. One of the ways to think about this bijection is as follows: 
each graded simple module $\mathrm{soc}(\Delta_e/\Delta_w)$, where $w\in\mathbf{B}_n$,
has multiplicity one in $\Delta_e$ (see Proposition~\ref{prop7}), moreover, the 
images of these $\mathrm{soc}(\Delta_e/\Delta_w)$ in the graded Grothendieck group of 
$\mathcal{O}_0$ are linearly independent.

Theorem~\ref{thm1} provides a categorical, or, alternatively,
a representation theoretic interpretation of the poset $\mathbf{B}_n$.
We note that this interpretation is completely different from the one in \cite{IRRT}.
The crucial step towards the formulation of Theorem~\ref{thm1}
was an accidental numerical observation which can be found 
in Corollary~\ref{cor8}.

For $x\in S_n$, denote by $\ell(x)$ the {\em length} of $x$
(as an element of a Coxeter group), and by $\mathbf{c}(x)$,
the {\em content} of $x$, that is
the number of different simple reflections appearing in a reduced
expression of $x$ (this number does not depend on the choice of a 
reduced expression). Denote by $\Phi:\mathbf{B}_n\to\mathcal{J}$
the map given by $\Phi(w)=w_{i,j}$, if $w\in \mathbf{B}_n$ is such that
$s_iw<w$ and $ws_j<w$. 
For $x\in S_n$, denote by $\mathbf{BM}(x)$ the
set of all Bruhat maximal elements in the set 
$\mathbf{B}(x):=\{y\in \mathbf{B}_n\,:\,y\leq x \}$.
For $w\in S_n$, we denote by $\nabla_x$ the {\em dual Verma module}
obtained from $\Delta_x$ by applying the simple preserving duality on
$\mathcal{O}_0$, see \cite{Hu}.

Theorem~\ref{thm1}, combined with \cite[Theorem~32]{Ma},
has the following homological consequence:

\begin{corollary}\label{cor2}
Let $x,y\in S_n$. Then we have
\begin{displaymath}
\dim\mathrm{Ext}_{\cO}^{1}(L_x,\Delta_y)=\dim\mathrm{Ext}_{\cO}^{1}(\nabla_y,L_x)=
\begin{cases}
\mathbf{c}(xy), & x=w_0;\\
1, & x\in \Phi(\mathbf{BM}(y));\\
0, & \text{otherwise}. 
\end{cases}
\end{displaymath}
\end{corollary}

Theorem~\ref{thm1} and Corollary~\ref{cor2} are proved in Section~\ref{s2}.
Both Theorem~\ref{thm1} and Corollary~\ref{cor2} extend to singular blocks of $\cO$, which we prove in Theorem~\ref{thm2I} and Proposition~\ref{socleI} in Section \ref{s7.2}.

The starting point of this paper was a question the second author received from Sascha Orlik and Matthias
Strauch, namely, whether,
for $x,y\in S_n$ such that $x<y$, the space
$\mathrm{Ext}_{\cO}^1(\Delta_y/\Delta_{w_0},\Delta_x)$ 
can be non-zero?

From Corollary~\ref{cor2}, it follows that the answer is ``yes''.
For example, for the algebra $\mathfrak{sl}_4$, one can take $y=s_2w_0$, 
in which case $\Delta_y/\Delta_{w_0}\cong L_y$, and 
$x=s_2$. By Corollary~\ref{cor2}, we have 
$\mathrm{Ext}_{\cO}^1(L_y,\Delta_x)=\mathbb{C}$.

\subsection{Description of additional results}\label{s1.4}

In Section \ref{s3}, we relate the main results with a number of combinatorial tools. 
In Subsections~\ref{s3.1} we explain how the combinatorics and the
Hasse diagram of the poset $\mathbf{B}_n$ appear naturally in our
representation theoretic interpretation of this poset.
In Subsections~\ref{s3.3} we establish representation theoretic 
significance of the notions of the essential set of a permutation and
the associated rank function. In particular, in Corollary~\ref{cor_ess}
we show that the socular constituents in the cokernel of
the inclusion of a Verma module indexed by $w\in S_n$
into the dominant Verma module are determined
by the essential set of $w$. Further, in Proposition~\ref{grsoc}
we show how the associated rank function can be used to compute the
degrees of these simple socular constituents in the graded picture.

\section{Proofs of the main results}\label{s2}

\subsection{Category $\mathcal{O}$ tools}\label{s2.1}

Denote by $\underline{\mathcal{O}}_0$ the full subcategory of ${\mathcal{O}}_0$
which consists of all modules on which the center $Z(\mathfrak{g})$ of the universal 
enveloping algebra $U(\mathfrak{g})$ acts diagonalizably. Note that all Verma modules
and all simple modules in ${\mathcal{O}}_0$ are objects in $\underline{\mathcal{O}}_0$.
By \cite{So}, the category $\underline{\mathcal{O}}_0$ has an auto-equivalence,
denoted by $\Theta$, which
maps $L_x$ to $L_{x^{-1}}$, for $x\in S_n$. Consequently, 
$\Theta(\Delta_x)\cong \Delta_{x^{-1}}$, for $x\in S_n$. We note that 
$\Theta$ does {\bf not} extend to the whole of ${\mathcal{O}}_0$.

For $x\in S_n$, we have the corresponding {\em twisting functor}
$T_x$ on $\mathcal{O}_0$, see \cite{AS}, and the corresponding
{\em shuffling functor} $C_x$ on $\mathcal{O}_0$, see \cite{Ca,MS}.
Both $\mathcal{L}T_x$ and $\mathcal{L}C_x$ are self-equivalences
of the bounded derived category $\mathcal{D}^b(\mathcal{O})$.
In the proof below we usually use twisting functors and $\Theta$, 
however, one can, alternatively, use twisting and shuffling functors
or shuffling functors and $\Theta$.

The category $\mathcal{O}_0$ is equivalent to the category of 
finite dimensional modules over a finite dimensional basic algebra, which we denote by $A$. 
The algebra $A$ is unique, up to isomorphism. By \cite{So2}, it is Koszul
and hence admits a Koszul $\mathbb{Z}$-grading.
We denote by ${}^{\mathbb{Z}}\mathcal{O}_0$ the category of
finite dimensional $\mathbb{Z}$-graded $A$-modules. We denote by $\langle k\rangle$ the grading shift on
${}^{\mathbb{Z}}\mathcal{O}_0$ normalized such that
it maps degree $k$ to degree zero, and we fix the standard 
graded lifts of $L_w$ concentrated in degree zero, and of $\Delta_w$ such that its top is
concentrated in degree zero.

\subsection{Potential socle of the cokernel of an inclusion of Verma modules}\label{s2.2}

\begin{proposition}\label{prop3}
Let $x,y,z\in S_n$ be such that $x\geq y$ and $L_z$ is in the socle of
$\Delta_y/\Delta_x$. Then $z\in \mathcal{J}$. 
\end{proposition}

\begin{proof}
As $\Delta_y/\Delta_x$ is a submodule of $\Delta_e/\Delta_x$, it is enough to prove
the claim in the case $y=e$.

Consider first the case $x=w_0$. Then $\Delta_{w_0}$ is the socle of $\Delta_e$ and
we claim that the socle of $\Delta_e/\Delta_{w_0}$ consists of all $L_{x}$, where $x\in W$ is such that
$\ell(x)=\ell(w_0)-1$. The best
way to argue that the socle of $\Delta_e/\Delta_{w_0}$ is as described above is
to recall that $\Delta_e$ is rigid, see \cite{Ir,BGS}. This means that the
socle filtration and the radical filtration of $\Delta_e$ coincide. 
Moreover,
these two filtrations also coincide with the graded filtration, see 
\cite[Proposition~2.4.1]{BGS}.
The submodule $\Delta_{w_0}=L_{w_0}$ in $\Delta_e$ lives in degree $\ell(w_0)$. By rigidity,
the socle of $\Delta_e/\Delta_{w_0}$ consists of simple modules which live in degree
$\ell(w_0)-1$. Since, by KL combinatorics, a simple module $L_x$ appears in
$\Delta_e$ only in degrees $\leq \ell(x)$ and $L_{w_0}$ has multiplicity one, the only simples in degree
$\ell(w_0)-1$ are those $L_x$ for which $\ell(x)=\ell(w_0)-1$. Note also that 
all such $x$ belong to $\mathcal{J}$. This proves the claim of the
proposition in the case $x=w_0$.

For $x\neq w_0$, using the self-equivalence $\mathscr{L}T_{w_0x^{-1}}$, we have
\begin{equation}\label{eq11}
\mathrm{Hom}_{\mathcal{D}^b(\mathcal{O})}(L_z,\Delta_e/\Delta_x)\cong
\mathrm{Hom}_{\mathcal{D}^b(\mathcal{O})}(\mathscr{L}T_{w_0x^{-1}}(L_z),
\Delta_{w_0x^{-1}}/\Delta_{w_0}). 
\end{equation}
By the previous two paragraphs, the socle of $\Delta_{w_0x^{-1}}/\Delta_{w_0}$
consist of the simples $L_w$, where $w\in\mathcal{J}$. Therefore,
for the right hand side of \eqref{eq11} to be non-zero, the module
$T_{w_0x^{-1}}(L_z)$ must contain some simple subquotient of the form
$L_w$, where $w\in\mathcal{J}$. From \cite[Theorem~6.3]{AS}, it follows that all 
simple subquotients of $T_{w_0x^{-1}}(L_z)$ are of the form
$L_u$, where $u\leq_{\mathtt{J}}z$. This yields $z\in \mathcal{J}$,
completing the proof.
\end{proof}

\begin{proposition}\label{prop4}
Let $x\in S_n$ and $s$ be a simple reflection.
\begin{enumerate}[$($i$)$]
\item \label{prop4.1} If $sx<x$, then the socle of $\Delta_e/\Delta_x$ contains
some $L_y$ such that $sy>y$.
\item \label{prop4.2} If $xs<x$, then the socle of $\Delta_e/\Delta_x$ contains
some $L_y$ such that $ys>y$.
\end{enumerate}
\end{proposition}

\begin{proof}
Due to the assumption $sx<x$, we have  $\Delta_x\subset \Delta_{sx}$,
in particular, the socle of $\Delta_e/\Delta_x$ contains the socle of 
$\Delta_{sx}/\Delta_x$. The module $\Delta_{sx}/\Delta_x$ is non-zero
and, by construction,
$s$-finite (i.e. the action on this module of the $\mathfrak{sl}_2$-subalgebra 
of $\mathfrak{g}$ corresponding to $s$ is locally finite). In particular,
any $L_y$ in the socle of $\Delta_{sx}/\Delta_x$ is also $s$-finite.
Therefore $sy>y$ for each $y$ such that $L_y$ is in the socle of
$\Delta_{sx}/\Delta_x$. Claim~\eqref{prop4.1} follows.
Claim~\eqref{prop4.2} follows from claim~\eqref{prop4.1} using $\Theta$.
\end{proof}

\begin{corollary}\label{cor5}
Let $w\in S_n$  be such that  the socle of $\Delta_e/\Delta_w$ is simple.
Then $w$ is a bigrassmannian permutation.
\end{corollary}

\begin{proof}
Assume that $s$ and $t$ are different simple reflections such that
$sw<w$ and $tw<w$. By Proposition~\ref{prop4}, the socle of
$\Delta_e/\Delta_w$ contains some $L_y$ such that $sy>y$ and
some $L_z$ such that $tz>z$. Both $y,z\in\mathcal{J}$. Note that, for each 
element $u$ of $\mathcal{J}$, there is at most one simple reflection $r$
such that $ru>u$. This means that $y\neq z$ and hence the socle of
$\Delta_e/\Delta_w$ is not simple. A similar argument works in the case
when there are different simple reflections $s$ and $t$ such that
$ws<w$ and $wt<w$. The claim follows.
\end{proof}

\begin{proposition}\label{prop6}
Assume that $x\in S_n$ and $y\in \mathcal{J}$ be such that 
$L_y$ is in the socle of $\Delta_e/\Delta_x$. Assume that 
$s$ is a simple reflection such that $sy>y$. Then $sx<x$.
\end{proposition}

\begin{proof}
Assume $sx>x$. Applying $\mathscr{L}T_s$ and using \cite[Theorem~6.1]{AS}, 
similarly to \eqref{eq11}, we have 
\begin{displaymath}
\mathrm{Hom}_{\mathcal{D}^b(\mathcal{O})}(L_y,\Delta_e/\Delta_x)\cong
\mathrm{Hom}_{\mathcal{D}^b(\mathcal{O})}(L_z[1],
\Delta_{s}/\Delta_{sx}). 
\end{displaymath} 
The right hand side of this equality is $0$ since $\Delta_{s}/\Delta_{sx}$
is a module in homological position $0$ while $L_z[1]$ is a module
in homological position $-1$. The claim follows.
\end{proof}

\subsection{Combinatorial tools}\label{s2.3}

Consider the Hecke algebra $\mathbb{H}_n$ associated to the Coxeter system $(S_n,S)$, 
where $S$ is the set $\{s_1,\cdots, s_{n-1}\}$ of simple reflections. 
$\mathbb{H}_n$ is a $\mathbb Z[v,v\inv]$-algebra 
generated by $H_i$, for $1\leq i\leq n-1$, which satisfy the braid relations 
\[H_iH_{i+1}H_i=H_{i+1}H_iH_{i+1},\]
for all $1\leq i\leq n-2$,
and the quadratic relation
\[(H_i+v)(H_i-v\inv)=0,\]
for all $1\leq i\leq n-1$.
Given a reduced expression $w=s_i s_j\cdots s_k$ of $w\in S_n$, we let $H_w=H_iH_j\cdots H_k$. 
The element $H_w$ is independent of the choice of the reduced expression, 
and $\{H_w\,:\,w\in S_n\}$ is a ($\mathbb Z[v,v\inv]$-)basis of $\mathbb H_n$ called the {\em standard basis}.
Consider the ($\mathbb Z$-algebra-)involution on $H_n$ denoted by  a bar, 
such that $\overline{v}= v\inv$ and $\overline{H_i} = H_i\inv$.
Then there is a unique element $\kl_w$ in $\mathbb H_n$ such that $\overline{\kl_w} = \kl_w$ and 
\[\kl_w = H_w + \sum_y p_{y,w}H_y,\]
for some $p_{y,w}\in v\mathbb Z[v]$.
The elements $\kl_w$, where $w\in W$, form a basis of $\mathbb H_n$ called the {\em Kazhdan-Lusztig (KL) basis}.
Given $w,y\in S_n$, the coefficient of $v$ in $p_{w,y}+p_{y,w}$ is denoted by 
$\mu(w,y)=\mu(y,w)$, defining the {\em (Kazhdan-Lusztig) $\mu$-function}.
If $s\in S$ then $\kl_s=H_s+v$
and we have
\begin{equation}\label{sy'}
\kl_s\kl_y = (v+v\inv)\kl_{y} ,   
\end{equation}
for $sy<y$, and 
\begin{equation}\label{sy}
    \kl_s\kl_y = \kl_{sy}+\sum_{sx<x, x<y}\mu(x,y)\kl_x,
\end{equation}
for $sy>y$.
Another basic fact is that
\begin{equation}\label{longest}
    p_{x,w_0}=v^{\ell(w_0)-\ell(x)}.
\end{equation}
For more details about KL basis we refer to \cite{KL}.

A consequence of the Kazhdan-Lusztig conjecture (which is now a theorem, see \cite{KL,BB,BK,EW}), 
is that the polynomial $p_{x,y}$ above describes the composition multiplicities of 
the Verma modules in $\!^\mathbb Z\cO_0$ in the following sense: 
\[\sum_{i\in\mathbb{Z}} [\Delta_w:L_y\langle - i\rangle]v^i = p_{w,y}.\]

From the Kazhdan-Lusztig conjecture/theorem, combined with the Koszul self-duality of
$\cO_0$ from \cite{So2,BGS}, it follows that the integer $\mu(w,y)$ equals the dimension of the 
first extension space between $L_w$ and $L_y$ (and hence determines the Gabriel quiver
of the algebra $A$).

We denote by $\mathbf{a} \colon W \to \mathbb{Z}_{\geq 0}$ Lusztig's $\mathbf{a}$-function, see \cite{Lu2}. The value $\mathbf{a}(w)$ does not change when $w$ varies over a two-sided cell. We write $\mathbf{a}(\mathcal J)=\mathbf{a}(w)$ for $w\in \mathcal J$. By definition, the value $\mathbf{a}(w)$ for $w \in \mathcal{J}$  describes the minimal possible 
degree shift for simple subquotients $L_u$, with $u\in \mathcal{J}$, of $\Delta_e$
(i.e. the degree shift for the top layer of the tetrahedron on Figure \ref{fig_tetra}).
The minimal shift is achieved exactly when $w$ is a Duflo element. 
It follows that 
\begin{equation}\label{pdegrees}
    p_{e,w}\in \mathbb Z \{ v^a\ |\ \mathbf{a}(w)\leq a\leq \ell(w)\}
\end{equation} 
where the first equality holds if and only if $w$ is a Duflo element.
Since $\mathcal{J}$ contains $w_{1,1}$, which is the longest element of some
parabolic subgroup of $S_n$, we have $\mathbf{a}(\mathcal{J})=\ell(w_{1,1})=\frac{(n-1)(n-2)}{2}$.

\begin{proposition}\label{muJ}
Let $x,y\in \jc$. Then $\mu(x,y)\neq 0$ if and only if $x,y$ are adjacent in the Bruhat graph.
In the latter case, we have $\mu(x,y)=1$.
\end{proposition}

\begin{proof}
By the Kazhdan-Lusztig inversion formula, see \cite{KL}, (or, equivalently, by Koszul duality,
see \cite{BGS}), we have $\mu(x,y)=\mu(w_0x\inv,w_0y\inv)$. Now, the claim follows from the 
similar fact about the small two-sided cell (see \cite[Proposition 3.8 (d)]{Lu} or \cite{KMMZ,KM}).
\end{proof}

For convenience of the reader, on Figure \ref{fig_bruhat_J} we give the Bruhat graph of $\mathcal{J}$.
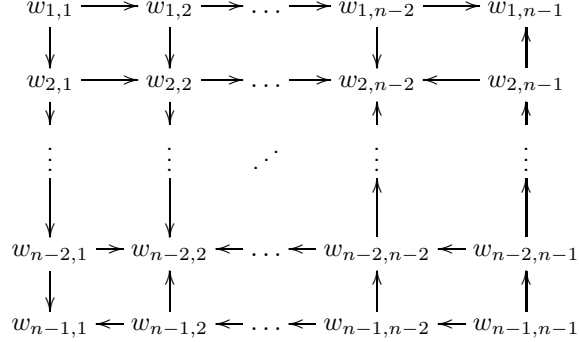
\begin{figure}[ht]
    \centering
\[ \xymatrix@C=1em@R=1.5em{
w_{1,1} \ar[r] \ar[d] & w_{1,2} \ar[r]\ar[d] & \ldots \ar[r] & w_{1,n-2} \ar[r]\ar[d] & w_{1,n-1} \\
w_{2,1} \ar[r]\ar[d] & w_{2,2} \ar[r]\ar[d] & \ldots \ar[r] & w_{2,n-2} & w_{2,n-1} \ar[l]\ar[u] \\  
\vdots \ar[d] & \vdots \ar[d] & \reflectbox{$\ddots$}  & \vdots \ar[u] & \vdots \ar[u] \\
w_{n-2,1} \ar[r]\ar[d] & w_{n-2,2}  & \ldots \ar[l] & w_{n-2,n-2} \ar[l]\ar[u] & w_{n-2,n-1} \ar[l]\ar[u]  \\
w_{n-1,1}  & w_{n-1,2} \ar[l]\ar[u]  & \ldots \ar[l] & w_{n-1,n-2} \ar[l]\ar[u] & w_{n-1,n-1} \ar[l] \ar[u]    }\]
    \caption{Bruhat graph of the penultimate cell (the rows are right, and the columns are left cells).}
    \label{fig_bruhat_J}
\end{figure}

\begin{lemma}\label{propp}
Suppose $y=w_{i,j}\in \jc$, with $i\neq j$. Then
\begin{equation}\label{pp}
p_{s,y} = v\inv p_{e,y},\quad \text{ for all } s \in S.
\end{equation} 
\end{lemma}

\begin{proof}
Suppose $s\neq s_i$. Then, comparing the $H_e$-coefficients in \eqref{sy'}, 
gives \eqref{pp}.
If $s=s_i$, then we have $s\neq s_j$, so comparing the $H_e$-coefficients 
in the equation $\kl_y\kl_s=(v+v\inv)\kl_y$ gives \eqref{pp}.
\end{proof}

\begin{lemma}\label{prop2p}
Suppose $y=w_{i,j}\in \jc$, with $i\neq j$, $s\in S$ and $sy>y$, we have
\begin{equation}\label{2p}
    (v+v\inv)p_{e,y}=p_{e,sy} + \sum_{\substack {u \in S \\ uy\sim_{\mathtt{L}} y \\ uy<y}} p_{e,uy}.
\end{equation}
\end{lemma}

\begin{proof}
Consider Equation~\eqref{sy}, for $y\in \jc$.
Since all $x$ appearing in this equation with non-zero coefficient
satisfy, at the same time, $x<y$ and  $x\geq_{\mathtt{L}} y$, 
and, since $y\in\jc$, we have $x\sim_{\mathtt{L}} y$. 
Now, comparing the $H_e$-coefficients on both sides in Equation~\eqref{sy}, we get
\begin{equation}\label{2p''}
    (v+v\inv)p_{e,y}=p_{s,y} + v p_{e,y}=p_{e,sy} + \sum_{\substack{x\sim_{\mathtt{L}} y \\ x<y}} \mu(x,y)p_{e,x},
\end{equation}
where we used \eqref{pp} for the first equality.
Then Proposition \ref{muJ} reduces \eqref{2p''} to \eqref{2p}.
\end{proof}
 
\begin{proposition}\label{p_e,w_I}
Let $y=w_{1,1}$ or $y=w_{n-1,n-1}$. Then $p_{e,y}=v^{\ell(y)}$.
\end{proposition}

\begin{proof}
The element $w_{n-1,n-1}$ is the longest element for the parabolic subgroup of 
$S_n$ generated by $\{1,\ldots, n-2\}$. The element $w_{1,1}$ is the longest 
element for the parabolic subgroup of $S_n$ generated by $\{2,\ldots,n-1\}$. 
Since the KL basis can be computed in the parabolic subgroups, the claim of the
lemma follows from \eqref{longest} applied to the Coxeter group $S_{n-1}$
with the corresponding identification of simple reflections.
\end{proof}

For $i,j \in \{ 0, \ldots,n\}$, put $p_{i,j} := p_{e,w_{i,j}}$ if $1\leq i,j \leq n-1$, and $p_{i,j} :=0$ otherwise.

\begin{lemma}
Suppose $w_{i,j}\in \jc$, with $i\neq j$. Then
\begin{equation} \label{eq_p4}
    (v+v^{-1})p_{i,j} =  p_{i-1,j} + p_{i+1,j} + \delta_{i,n-j} \cdot v^{\ell(w_0)},
\end{equation}
where $\delta$ denotes the Kronecker symbol.
\end{lemma}
\begin{proof}
This follows from (\ref{2p}) with $s=s_i$, (\ref{longest}) and Figure \ref{fig_bruhat_J}.
\end{proof}

\begin{proposition}\label{prop7}
Let $w=w_{i,j}\in \jc$. 
We have 
\begin{equation}\label{want}
    p_{e,w}=v^{\ell(w)}+v^{\ell(w)-2}+\cdots +v^{\ell(w)-2d(w)},
\end{equation}
where 
\begin{equation*}
    d(w)=\operatorname{min}\{i-1,j-1, n-1-i,n-1-j\}.
\end{equation*}
\end{proposition}

\begin{proof}
We proceed by induction on $d=d(w)$.

Let $d=0$. By symmetry (flipping the Dynkin diagram on one hand, and taking inverses of permutations on the other hand), it is enough  to consider $w=w_{i,1}$. 
If $i=1$, then Proposition \ref{p_e,w_I} gives \eqref{want}. If $i=2$, then (\ref{eq_p4}) gives \[ (v+v\inv)p_{2,1}= v^{a} + p_{3,1}, \] where $a:=\mathbf{a}(\mathcal{J})$. Since $\ell(w_{2,1})=a+1$, by \eqref{pdegrees}, we have $p_{2,1}=c v^{a+1}$ for some $c \in \mathbb Z$. For the same reason, $p_{3,1}$ does not have term proportional to $v^a$, so we get $p_{2,1} = v^{a+1}$ and $p_{3,1} = v^{a+2}$. From this we can obtain \eqref{want} for any $w_{i,1}$, using (\ref{eq_p4}) and a two-step induction.

% Not necessary:
%For the case $d=1$, again by symmetry it is enough to consider $w_{2,j}$, for $j = 2,\ldots,n-2$. It is covered by \eqref{eq_p4} for $w_{1,j}$, and the already proved $d=0$ case.

Let $d = d(w)>0$. By symmetry, we may assume $w=w_{d+1,j}$, with $d<j<n-d$. We assume that (\ref{want}) is true for $w_{d,1}$ and (if exists) for $w_{d-1,1}$. Equation (\ref{eq_p4}) applied to $w=w_{d,j}$ gives
\[ (v+v^{-1})p_{d,j} =  p_{d-1,j} + p_{d+1,j}. \]
From this, it easily follows that (\ref{want}) is also true for $w_{d+1,j}$.
%
% Old version:
%Let $d(w)>1$. Then both $d(w)$ and $d(rw)$, for all $r\in S$, are nonzero. Therefore, we can either find $u,t\in S$ such that $d(uw)=d(w)-1$ and $d(tuw)=d(w)-2$, or find $u,t\in S$ such that $d(wu)=d(w)-1$ and $d(wut)=d(w)-2$.  By symmetry (for example, applying $\Theta$ in the beginning of \S\ref{s2.1}),  it is enough to consider the former case. Equation~\eqref{2p} now gives \begin{equation}\label{puy} (v+v\inv)p_{e,uw}=p_{e,w}+p_{e,tuw}. \end{equation} Using induction hypothesis, the terms $p_{e,uw}$ and $p_{e,tuw}$ in \eqref{puy} satisfy \eqref{want}, which solves \eqref{puy} to give the desired formula \eqref{want}.
%
\end{proof}

For $i,j\in\{1,2,\dots,n-1\}$, we denote by $\mathbf{B}_n^{(i,j)}$ the set of all
$w\in \mathbf{B}_n$ such that $s_iw<w$ and $ws_j<w$.

\begin{corollary}\label{cor8}
We have 
\begin{equation}\label{eq14}
\displaystyle|\mathbf{B}_n|=\sum_{w\in\mathcal{J}}p_{e,w}(1),
\end{equation}
moreover, for each $i,j\in\{1,2,\dots,n-1\}$, we have
$|\mathbf{B}_n^{(i,j)}|=p_{e,w_{i,j}}(1)$. 
\end{corollary}

\begin{proof}
By Proposition~\ref{prop7}, the right hand side of \eqref{eq14}
evaluates to the tetrahedral number $\frac{(n-1)n(n+1)}{6}$. That the left hand
side of \eqref{eq14} is given by the same number follows from the main result of \cite{Ko2}.
The refined version, for each $i,j\in\{1,2,\dots,n-1\}$, follows from
the main result of \cite{EH}. Both statements are best seen by comparing the
main result of \cite{EH} to the picture described in Subsection~\ref{s3.1}, 
where the right hand side \eqref{eq14}
is realized as a tetrahedron in $\mathbb{R}^3$ and the elements corresponding to
$\mathbf{B}_n^{(i,j)}$ consist of all elements of this tetrahedron aligned along
a vertical line.
\end{proof}

\begin{proposition}\label{prop9}
Let $i,j\in\{1,2,\dots,n-1\}$. The restriction of the Bruhat order to each 
$\mathbf{B}_n^{(i,j)}$ is a chain, moreover, for $x,y\in \mathbf{B}_n^{(i,j)}$
such that $x<y$, there exist $w\in \mathbf{B}_n\setminus\mathbf{B}_n^{(i,j)}$
such that $x<w<y$.
\end{proposition}

\begin{proof}
This follows directly from the main result of \cite{EH}.
The poset $\mathbf{B}_n$ can be realized as a tetrahedron in
$\mathbb{R}^3$, see Subsection~\ref{s3.1}. The subset
$\mathbf{B}_n^{(i,j)}$ consist of all elements of this tetrahedron aligned along
a vertical line. If such a line contains more than one element of 
the tetrahedron, we have a diamond shaped subset of $\mathbf{B}_n$
and the elements of this diamond outside the original line provide
the necessary $w$, see for example the points $(2,2,5)$ and $(2,2,3)$, in the case $n=4$, where $w$ can be chosen as any of the elements
$(2,1,4)$, $(1,2,4)$, $(3,2,4)$ or $(2,3,4)$.
\end{proof}

\subsection{Proof of Theorem~\ref{thm1}\eqref{thm1.1}, \eqref{thm1.2}}\label{s2.4}

Corollary~\ref{cor5} gives one direction of \eqref{thm1.1}. For the other direction, we need for any
$w\in\mathbf{B}_n$, the module $\Delta_e/\Delta_w$ has simple socle.
Assume that $s_i$ and $s_j$ are such that $s_iw<w$ and $ws_j<w$. 
Then, by Proposition~\ref{prop6}, $L_{w_{i,j}}$ is the only possible 
simple subquotient in the socle of $\Delta_e/\Delta_w$. We need to 
prove that the multiplicity of $L_{w_{i,j}}$ in the socle of 
$\Delta_e/\Delta_w$ equals $1$.

By Proposition~\ref{prop9}, there is a chain 
\begin{displaymath}
w_1<u_1<w_2<u_2<\dots<w_{k-1}<u_{k-1}<w_k, 
\end{displaymath}
where $\mathbf{B}_n^{(i,j)}=\{w_1,w_2,\dots,w_k\}$
and all $u_i\in \mathbf{B}_n\setminus \mathbf{B}_n^{(i,j)}$. This chain gives an a sequence of inclusions
\begin{displaymath}
\Delta_e\supsetneq \Delta_{w_1}\supsetneq \Delta_{u_1}\supsetneq 
\Delta_{w_2}\supsetneq \Delta_{u_2}
\supsetneq\dots\supsetneq \Delta_{w_k}
\end{displaymath}
which, in turn, gives rise to a sequence of projections
\begin{displaymath}
\Delta_e/\Delta_{w_1}\twoheadleftarrow
\Delta_e/\Delta_{u_1}\twoheadleftarrow
\Delta_e/\Delta_{w_2}\twoheadleftarrow
\Delta_e/\Delta_{u_2}\twoheadleftarrow\dots
\twoheadleftarrow\Delta_e/\Delta_{w_k}.
\end{displaymath}
This implies that the socle of each $\Delta_e/\Delta_{w_i}$ contains a summand
which is not a summand of the socle of $\Delta_e/\Delta_{w_j}$, for any $j<i$.
From Proposition~\ref{prop7} and Corollary~\ref{cor8} we obtain that 
$k=P_{e,w_{i,j}}(1)$ and hence the socle of each $\Delta_e/\Delta_{w_i}$
contains a unique summand which is not in the socle of any 
$\Delta_e/\Delta_{w_j}$, for $j<i$. 

Therefore, for both claims \eqref{thm1.1} and  \eqref{thm1.2}, 
it is enough to prove that the socle of 
each $\Delta_e/\Delta_{w_i}$ is simple. We
argue that the socle constituents of any $\Delta_e/\Delta_{w_j}$, where $j<i$,
cannot contribute to the socle of $\Delta_e/\Delta_{w_i}$. Assume that this is not
the case and some socle constituent of some $\Delta_e/\Delta_{w_j}$ contributes
to the socle of $\Delta_e/\Delta_{w_i}$. Then it also must contribute to the socle
of $\Delta_e/\Delta_{u_{i-1}}$, since $w_j<u_{i-1}<w_i$. But this contradicts
Proposition~\ref{prop6} and completes the proof.

\subsection{Proof of Theorem~\ref{thm1}\eqref{thm1.3} and \eqref{thm1.4}}\label{s2.6}

Both the statements follow from the bijection given in Theorem~\ref{thm1}\eqref{thm1.2}
and the fact that $\Delta_x\subset \Delta_y$ is equivalent to $x\geq y$.

\subsection{Proof of Corollary~\ref{cor2}}\label{s2.9}

First of all, we note that the equality
\begin{displaymath}
\dim\mathrm{Ext}_{\cO}^{1}(L_x,\Delta_y)=\dim\mathrm{Ext}_{\cO}^{1}(\nabla_y,L_x)
\end{displaymath}
follows by applying the simple preserving duality on $\mathcal{O}_0$.
Therefore we only need to prove the following: 
\begin{equation}\label{eq1}
\dim\mathrm{Ext}_{\cO}^{1}(L_x,\Delta_y)=
\begin{cases}
\mathbf{c}(xy), & x=w_0;\\
1, & x\in \Phi(\mathbf{BM}(y));\\
0, & \text{otherwise}. 
\end{cases}
\end{equation}

Consider first the case $x=w_0$. In this case $L_{w_0}=\Delta_{w_0}$.
Applying the twisting functor $T_{y}$ and using the fact that 
it is acyclic on Verma modules, see \cite[Theorem~2.2]{AS}, we observe that
\begin{displaymath}
\mathrm{Ext}_{\cO}^{1}(\Delta_{y^{-1}w_0},\Delta_e)=
\mathrm{Ext}_{\cO}^{1}(\Delta_{w_0},\Delta_y).
\end{displaymath}
By  \cite[Theorem~32]{Ma}, the dimension of $\mathrm{Ext}_{\cO}^{1}(\Delta_{y^{-1}w_0},\Delta_e)$
equals $\mathbf{c}(y^{-1}w_0)=\mathbf{c}(xy)$. This establishes
\eqref{eq1} in the case $x=w_0$.

Assume now that $x\neq w_0$. Let 
\begin{equation}\label{eq2}
0\to \Delta_y\to M \to L_x \to 0 
\end{equation}
be a non-split short exact sequence. Since $L_x$ is a simple object and
\eqref{eq2} is non-split, the socle of $M$ coincides with the socle of $\Delta_y$
and hence is isomorphic to $L_{w_0}$. In particular, $M$ is a submodule of the
injective envelope $I_{w_0}$ of $L_{w_0}$. As the multiplicity of 
$L_{w_0}$ in $M$ is one, all nilpotent endomorphisms of $I_{w_0}$ send $M$ to $0$. 
Since $\Delta_e$, being a projective object in $\cO$, is copresented by projective-injective objects in $\cO$ (see \cite[\S~3.1]{KSX}) and since $I_{w_0}$ is the only indecomposable projective-injective object in $\cO_0$, it follows that $M$ is a submodule of $\Delta_e$.
 This means that
$L_x\cong M/\Delta_y$ corresponds to a socle constituent of $\Delta_e/\Delta_y$.
Therefore, for $x\neq w_0$,
formula \eqref{eq1} follows from Theorem~\ref{thm1}\eqref{thm1.4}. 

\section{Singular blocks}\label{s7.2}

In this subsection we extend our results on $\cO_0$ to an arbitrary block $\cO_\mu$ (that is, the Serre subcategory of $\cO$ with the simple objects $L(w\cdot \mu)$, for $w$ in the
$\mu$-integral subgroup $S_n^{(\mu)}$ of $S_n$) thus to the entire category $\cO$.
Let $\mu\in\mathfrak{h}^*$ be an $S_n^{(\mu)}$-dominant weight.
The first, standard, remark is that, up to equivalence with a block for some
other $n$, it is enough to consider the case when
$\mu$ is integral and hence $S_n^{(\mu)}=S_n$, see Soergel's combinatorial description of
$\mathcal{O}$ in \cite{So2}. Therefore, from now on we assume $\mu$ to be integral.

Note that, if $I= \{i\in\{1,\cdots,n-1\}\ |\  s_i\cdot\mu=\mu\}$ 
and $S_I$ is the parabolic subgroup of $S_n$ generated by $I$, then $S_I\cong S_\mu:=\operatorname{Stab}_W(\mu)$. 
So, we have $L(w\cdot\mu)\cong L(wz\cdot\mu)$ and $\Delta(w\cdot\mu)\cong\Delta(wz\cdot\mu)$, for any $z\in S_\mu$.
Given $w\in S_n$, we denote by $\underline{w}$ and $\overline{w}$ the unique 
shortest and the unique longest coset representatives for the coset $wS_\mu$ in $S_n/S_\mu$.
We have the indecomposable projective functors
\[\theta^0_\mu:\cO_0\to \cO_\mu\qquad\text{ and }\qquad \theta^\mu_0:\cO_\mu\to\cO_0\]
called {\em translation to the $\mu$-wall} and
{\em translation out of the $\mu$-wall}, respectively. (These functors are sometimes denoted by $T_0^\mu$ and $T_\mu^0$, respectively, e.g., in \cite{Hu}.)
They are biadjoint, exact and determined uniquely, up to isomorphism, by
$\theta^0_{\mu}\Delta_e=\Delta_\mu$ and $\theta_0^{\mu}\Delta_\mu=P_{\overline{e}}$,
respectively, see \cite{BG}. In particular, from the biadjointness it follows
that the action of the functor $\theta^0_\mu$ on simple
modules is given, for $w\in S_n$, by
\begin{displaymath}
\theta^0_\mu L_w\cong
\begin{cases}
L(w\cdot \mu),& w=\overline{w};\\
0,&\text{otherwise}.
\end{cases}
\end{displaymath}
On the level of the Grothendieck group, the functor $\theta_0^{\mu}\theta^0_{\mu}$
corresponds to the right multiplication with the sum of all elements in $S_\mu$,
see \cite[\S~3.4]{BG}.

\begin{proposition}\label{socleI}
Let $x,y\in S_n$ be such that $x<y$. 
Then we have \[\operatorname{soc}(\Delta(x\cdot\mu)/\Delta(y\cdot\mu)) \cong 
\theta^0_\mu(\operatorname{soc}\Delta_{\underline{x}}/\Delta_{\underline{y}}).\]
\end{proposition}

\begin{proof}
We may assume $x=\underline x$ and $y=\underline y$.

Suppose $L_w$ is a socle component of $\Delta_x/\Delta_y$. 
Then, by exactness, the module $\theta^0_\mu L_w$, whenever it is nonzero, that is,
whenever $w=\overline{w}$, is a socle component of 
\[\theta^0_\mu (\Delta_x/\Delta_y) \cong (\theta^0_\mu\Delta_x)/(\theta^0_\mu\Delta_y)
\cong \Delta(x\cdot\mu)/\Delta(y\cdot\mu).
\]

Now, assume that $L(w\cdot\mu)$ is a socle component of $\Delta(x\cdot\mu)/\Delta(y\cdot\mu)$.  
Then we have
\[\Ext_{\cO}^1(L(w\cdot \mu),\Delta(y\cdot\mu))\neq 0.\]
By adjunction, we also have
\[\Ext_{\cO}^1(L(w\cdot \mu),\Delta(y\cdot\mu))\cong \Ext_{\cO}^1(\theta^0_\mu L_{\overline{w}},\Delta(y\cdot\mu))\cong \Ext_{\cO}^1(L_{\overline{w}},\theta^\mu_0\Delta(y\cdot\mu)).\]
As $\theta^\mu_0$ is a projective functor, $\theta^\mu_0\Delta(y\cdot\mu)$ has a Verma flag.
As, on the level of the Grothendieck group, $\theta_0^{\mu}\theta^0_{\mu}$
corresponds to the right multiplication with the sum of all elements in $S_\mu$,
the Verma flag of  $\theta^\mu_0\Delta(y\cdot\mu)$ has
subquotients $\Delta_{yz}$, where $z\in S_\mu$, each appearing with multiplicity one.
It follows that $\Ext_{\cO}^1(L_{\overline{w}},\Delta_{yz})\neq 0$, for some $z\in S_\mu$, which means  $\overline{w}\in\mathcal J$ by Corollary~\ref{cor2}.
Moreover, we obtain that $L_{\overline{w}}$ appears in the socle of $\Delta_e/\Delta_{yz}$.
Since $\theta^0_\mu L_{\overline{w}}\neq 0$ while $\theta^0_\mu(\Delta_{y}/\Delta_{yz})=0$, the subquotient $L_{\overline{w}}$ is also contained in the socle of $\Delta_e/\Delta_{y}$.
To see that $L_{\overline{w}}$ is in the socle of $\Delta_x/\Delta_y$, we note that $\Delta_e/\Delta_x$ cannot contain this subquotient in the socle, for in that case $\theta^0_\mu(\Delta_e/\Delta_x)\cong \Delta(\mu)/\Delta(x\cdot\mu)$ would contain $L(w\cdot\mu)$ in the socle, contradicting our assumption (note that the graded multiplicities of simples from $\mathcal J$ in $\Delta_e$, as well as of the corresponding translated simples in 
$\Delta(\mu)=\theta^0_\mu\Delta_e$  are one or zero, so we can distinguish all such simple subquotients).
The proof is complete.
\end{proof}

\begin{theorem}\label{thm2I}
Let $x,y\in S_n$ and let $\mu$ be an integral, dominant weight. Then we have
\begin{displaymath}
\dim\mathrm{Ext}_{\cO}^{1}(L(x\cdot \mu),\Delta(y\cdot \mu))=\dim\mathrm{Ext}_{\cO}^{1}(\nabla(y\cdot\mu),L(x\cdot\mu))=
\begin{cases}
\mathbf{c}(\overline{x}\underline{y})-|I|, & \overline{x}=w_0;\\
1, & \overline{x}\in \Phi(\mathbf{BM}(\underline{y}));\\
0, & \text{otherwise}. 
\end{cases}
\end{displaymath}
\end{theorem}

\begin{proof}
Similarly to the proof of Corollary~\ref{cor2},
the proof of Proposition~\ref{socleI} takes care of the case $\overline{x}\neq w_0$. 

For the case $\overline{x}=w_0$, we need to generalize 
\cite[Theorem~32]{Ma}, which is proved in the regular setup, to
singular blocks. So, let $x=w_0$.
We may assume $y=\underline{y}$.
For the proof, we work in the graded category $\!^{\mathbb Z}\mathcal O$, where hom- and ext-functors are, as usual, denoted by $\hom$ and $\ext_\mathcal{O}^i$.
The precise claim in this case is 
\begin{equation}\label{DD}
    \dim\ext_{\cO}^1(L(w_0\cdot\mu),\Delta(y\cdot\mu)\langle i\rangle)=\begin{cases}\mathbf{c}(w_0y)-|I|, &\text{ if $i = \ell(w_0y)-2$};\\
    0, &\text{ otherwise.}
    \end{cases}
\end{equation}
If $\mu=0$, then, similarly to the proof of Corollary~\ref{cor2}, Formula~\eqref{DD}
reduces to \cite[Theorem 32]{Ma} using twisting functors.

In the general case, we first note that $L(w_0\cdot\mu)\cong \theta^0_\mu L_{w_0}$ 
where we now use the graded translation functor 
$\theta^0_\mu:\!^{\mathbb Z}\mathcal O_0\to \!^{\mathbb Z}\mathcal O_\mu$, 
see \cite{St}. By adjunction, this gives
\[\ext_{\cO}^1(L(w_0\cdot\mu),\Delta(y\cdot\mu)\langle i\rangle)\cong 
\ext_{\cO}^1(L_{w_0},\theta^\mu_0\Delta(y\cdot\mu)\langle i\rangle).\]
The graded module 
$\theta^\mu_0\Delta(y\cdot\mu)$
has a filtration
\begin{equation}\label{eqeq652}
0=M_{-1}\subset M_0\subset M_1\subset 
\dots\subset M_{\ell(\overline{e})}=\theta_\mu^0\Delta(y\cdot\mu), 
\end{equation}
with subquotients
\begin{displaymath}
M_i/M_{i-1}\cong \bigoplus_{z\in S_\mu:\ell(z)=i} \Delta_{yz}\langle \ell(z) \rangle,
\qquad\text{ for } i=0,1,2,\dots,\ell(\overline{e}).
\end{displaymath}
In particular, we have the short exact sequence
\begin{equation}\label{Tses}
    0\to M_0\to \theta^\mu_0\Delta(y\cdot\mu)\to N\to 0,
\end{equation}
where $M_0\cong \Delta_y$ and $N$ has an Verma flag induced from \eqref{eqeq652}. 
Consider the long exact sequence obtained by applying $\hom(L_{w_0},{}_-\langle i\rangle)$ to \eqref{Tses}, for $i\in\mathbb Z$.
By \eqref{DD} applied to $\mu=0$ and $i=\ell(w_0y)-2$, each of the spaces
$\ext_{\cO}^1(L_{w_0},\Delta_{yz}\langle \ell(w_0y)-2+\ell(z) \rangle)$ is zero and thus $\ext_{\cO}^1(L_{w_0},N\langle \ell(w_0y)-2 \rangle)=0$. 
Since $\theta^\mu_0\Delta(y\cdot\mu)$ has simple socle
and this socle lives in degree $\ell(w_0y)$, we also have
$\hom(L_{w_0},\theta^\mu_0\Delta(y\cdot\mu)\langle \ell(w_0y)-2\rangle)=0$.
Therefore, our long exact sequence gives the short exact sequence
\[0\to \hom(L_{w_0},N\langle \ell(w_0y)-2 \rangle)\to\ext_{\cO}^1(L_{w_0},\Delta_y\langle \ell(w_0y)-2 \rangle) \to\ext_{\cO}^1(L_{w_0}, \theta^\mu_0\Delta(y\cdot\mu)\langle \ell(w_0y)-2 \rangle) \to 0.\]
Using \eqref{DD}, for $\mu=0$, again, the dimension of the middle space 
in this short exact sequence is given by $\mathbf{c}(w_0y)$. 
Since the socle of $\Delta_{yz}$ in $\!^\mathbb{Z}\cO_0$ is $L_{w_0}\langle-\ell(w_0yz) \rangle$, we have 
\[\hom(L_{w_0},N\langle \ell(w_0y)-2 \rangle)= \hom(L_{w_0},\bigoplus_{s\in I}\Delta_{ys}\langle\ell(s) +\ell(w_0y)-2\rangle)=\hom(L_{w_0},\bigoplus_{s\in I}\Delta_{ys}\langle \ell(w_0ys)\rangle) =\mathbb C^{|I|}.\]
It follows that 
\[\dim\ext_{\cO}^1(L_{w_0}, \theta^\mu_0\Delta(y\cdot\mu)\langle \ell(w_0y)-2 \rangle)=\mathbf{c}(w_0y)-|I|.\]

It remains to show that, for $i\neq \ell(w_0y)-2$, we have
\[\ext_{\cO}^1(L_{w_0}, \theta^\mu_0\Delta(y\cdot\mu)\langle i \rangle)=0.\] 
For this, we consider the short exact sequence
\[0\to \theta^\mu_0\Delta(y\cdot\mu)\langle \ell(w_0y) \rangle\to I_{w_0}\to M\to 0\]
which induces
\[\hom(L_{w_0},I_{w_0}\langle -\ell(w_0y)+i\rangle )\to\hom(L_{w_0}, M\langle -\ell(w_0y)+i\rangle)\to
\ext_{\cO}^1(L_{w_0},  \theta^\mu_0\Delta(y\cdot\mu)\langle i \rangle) \to 0.\]
Thus, it remains to show that 
$\hom(L_{w_0}, M\langle j \rangle)=0$, for $j=-\ell(w_0y)+i \neq -2$. 
We denote by $K$ the cokernel of the inclusion 
$\theta^\mu_0\Delta(y\cdot\mu)\langle\ell(w_0y)\rangle \subset \theta^\mu_0\Delta(\mu)\langle\ell(w_0)\rangle$. In particular, the multiplicity of any shift of $L_{w_0}$ in
$K$ is zero. 
Then we have a canonical map $K\inj M$ by definition of $K$ and $M$. This defines a short exact sequence
\[0 \to K \to M\to Q\to 0,\]
where $Q$ has a filtration by $\Delta_{w}\langle\ell(w_0)+\ell(w) \rangle$, for each $w\in S_n\setminus S_I$. 

We claim that the socle of $Q$ comes from the socle of 
$\Delta_s\langle\ell(w_0)+1 \rangle$, for $s\in S\setminus I$. Indeed, 
The module $I_{w_0}$ is Koszul dual to the projective resolution of $L_{e}$.
The latter can be obtained from the BGG resolution of $L_{e}$, given in terms
of Verma modules, by gluing projective resolutions of Verma constituents
of the BGG resolution into a projective resolution of $L_e$. In this way, 
the Verma modules in the BGG resolution give rise to the subquotients
of the dual Verma filtration
of $I_{w_0}$. Note that any  $\Delta_x$ appears in the BGG resolution once and,
moreover, for any $y$ such that $\ell(y)=\ell(x)+1$ and $y>x$, the map
$\Delta_y\langle 1\rangle\to \Delta_x$ in the BGG resolution is non-zero, 
see \cite{BGG0}. The Koszul dual 
of this property is that the corresponding subquotient of the 
dual Verma flag of $I_{w_0}$ is
given by a nonzero morphism in the derived category and thus
by a non-split short exact sequence 
\begin{displaymath}
0\to \nabla_{w_0x^{-1}}\langle -1\rangle\to R\to \nabla_{w_0y^{-1}}\to 0. 
\end{displaymath}
Applying the simple preserving duality $\star$ on $\mathcal{O}$, we consider
the dual short exact sequence
\begin{equation}\label{eqeq987}
0\to \Delta_{w_0y^{-1}}\to R^\star\to \Delta_{w_0x^{-1}}\langle 1\rangle\to 0. 
\end{equation}
Write $w_0x^{-1}=u_1tu_2$ reduced, where $t$ is a simple reflection and $w_0y^{-1}=u_1u_2$. 
Then we can consider \eqref{eqeq987} as the image, under the
$u_2$-shuffling, of a non-split short exact sequence
\begin{equation}\label{eqeq988}
0\to \Delta_{u_1}\to R_1\to \Delta_{u_1t}\langle 1\rangle\to 0. 
\end{equation}
If $w_0=u_3u_1t$ is reduced, then, twisting \eqref{eqeq988} by $u_3$, we obtain
a non-split short exact sequence
\begin{equation}\label{eqeq989}
0\to \Delta_{w_0t}\to R_2\to \Delta_{w_0}\langle 1\rangle\to 0. 
\end{equation}
Here $\Delta_{w_0}=L_{w_0}$ and hence the fact that \eqref{eqeq989} is non-split
means that $R_2$ has simple socle. In particular, the action of the center of 
$\mathcal{O}_0$ on $R_2$ is not semi-simple.
This implies that  the action of the center of 
$\mathcal{O}_0$ on both $R_1$ and $R^\star$ is not semi-simple either and hence
they both have simple socle. Since $I_{w_0}$ is self-dual with respect to $\star$, 
this implies that the socle of $Q$ comes from its Bruhat 
minimal subquotients $\Delta_s\langle\ell(w_0)+1 \rangle$, where $s\in S\setminus I$.

Applying $\hom(L_{w_0},-)$, we obtain
\[\hom(L_{w_0},M \langle j \rangle)\hookrightarrow \hom(L_{w_0}, Q\langle j \rangle).\]
However,
\[\hom(L_{w_0}, Q\langle j \rangle)=\hom(L_{w_0}, \bigoplus_{s\in S\setminus I}\Delta_s \langle \ell(w_0)+1 +j \rangle)=0,\] for $j\neq -2$.
This completes the proof.
\end{proof}

\section{Representation theory versus Combinatorics}\label{s3}

\subsection{Tetrahedron}\label{s3.1}

Our computation for KL polynomials for $\mathcal{J}$ in Proposition~\ref{prop7}
gives a nice geometric diagram for simple subquotients of
$\Delta_e$ of the form $L_w\langle - i\rangle$, where $w\in\mathcal{J}$
and $i\in\mathbb{Z}$. By representing each simple subquotient $L_{w_{i,j}}\langle - k\rangle$ of $\Delta_e$ as the point $(i,j,k)$, we get a tetrahedron in $\mathbb{Z}^3$. In this picture, we join two points if the corresponding subquotients extend in $\Delta_e$ 
(the existence of the extension follows from the arguments in Subsection~\ref{s2.4}). 
In Figure~\ref{fig_tetra}
we show explicit examples for $n=3,4,5$ (note that, as usual in depicting positively
graded algebras, the $z$-axis is reversed in the pictures, 
so that the bottom of the picture consists of all elements of the form $sw_0$, for $s$ a simple reflection).
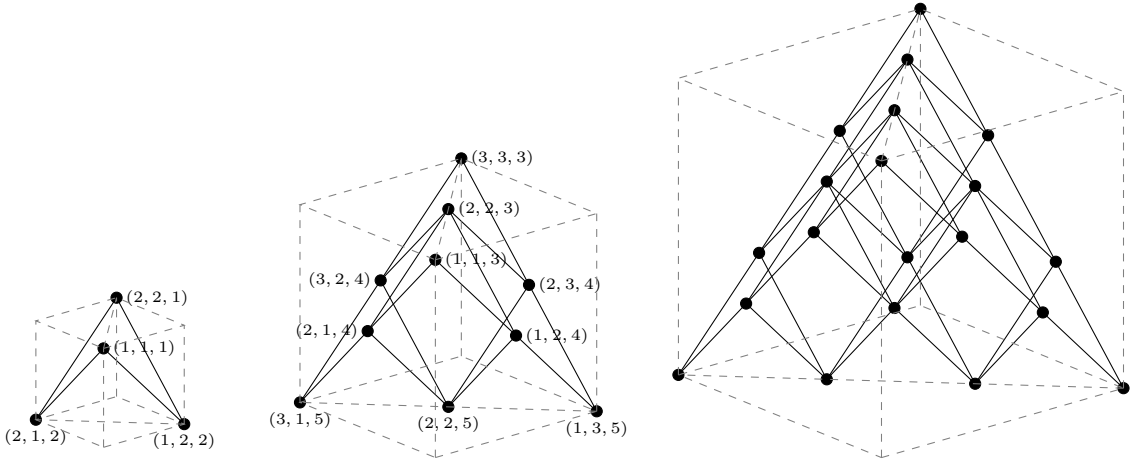
\begin{figure}[ht]
\centering
% Tetrahedron, n=3
\tdplotsetmaincoords{110}{130} 
\begin{tikzpicture}[tdplot_main_coords, scale=1.4]

\filldraw[black] (1, 1, -1) circle (1.5pt) node[anchor=west] {\tiny $(1, 1, 1)$};
\filldraw[black] (1, 2, -2) circle (1.5pt) node[anchor=north] {\tiny $(1, 2, 2)$};
\filldraw[black] (2, 1, -2) circle (1.5pt) node[anchor=north] {\tiny $(2, 1, 2)$};
\filldraw[black] (2, 2, -1) circle (1.5pt) node[anchor=west] {\tiny $(2, 2, 1)$};

\draw (1, 1, -1) -- (1, 2, -2);
\draw (1, 1, -1) -- (2, 1, -2);
\draw (2, 2, -1) -- (1, 2, -2);
\draw (2, 2, -1) -- (2, 1, -2);

\draw[gray, dashed] (1, 1, -2) -- (1, 2, -2) -- (2, 2, -2) -- (2, 1, -2) -- cycle;
\draw[gray, dashed] (1, 1, -1) -- (1, 2, -1) -- (2, 2, -1) -- (2, 1, -1) -- cycle;
\draw[gray, dashed] (1, 1, -2) -- (1, 1, -1);
\draw[gray, dashed] (1, 2, -2) -- (1, 2, -1);
\draw[gray, dashed] (2, 2, -2) -- (2, 2, -1);
\draw[gray, dashed] (2, 1, -2) -- (2, 1, -1);
\draw[gray, dashed] (1, 2, -2) -- (2, 1, -2);
\draw[gray, dashed] (1, 1, -1) -- (2, 2, -1);

\end{tikzpicture}
%
% Tetrahedron, n=4
\tdplotsetmaincoords{110}{130} 
\begin{tikzpicture}[tdplot_main_coords, scale=1.4]

\filldraw[black] (1, 1, -3) circle (1.5pt) node[anchor=west] {\tiny $(1, 1, 3)$};
\filldraw[black] (1, 2, -4) circle (1.5pt) node[anchor=west] {\tiny $(1, 2, 4)$};
\filldraw[black] (1, 3, -5) circle (1.5pt) node[anchor=north] {\tiny $(1, 3, 5)$};
\filldraw[black] (2, 1, -4) circle (1.5pt) node[anchor=east] {\tiny $(2, 1, 4)$};
\filldraw[black] (2, 2, -5) circle (1.5pt) node[anchor=north] {\tiny $(2, 2, 5)$};
\filldraw[black] (2, 2, -3) circle (1.5pt) node[anchor=west] {\tiny $(2, 2, 3)$};
\filldraw[black] (2, 3, -4) circle (1.5pt) node[anchor=west] {\tiny $(2, 3, 4)$};
\filldraw[black] (3, 1, -5) circle (1.5pt) node[anchor=north] {\tiny $(3, 1, 5)$};
\filldraw[black] (3, 2, -4) circle (1.5pt) node[anchor=east] {\tiny $(3, 2, 4)$};
\filldraw[black] (3, 3, -3) circle (1.5pt) node[anchor=west] {\tiny $(3, 3, 3)$};

\draw (1, 1, -3) -- (1, 2, -4);
\draw (1, 1, -3) -- (2, 1, -4);
\draw (1, 2, -4) -- (1, 3, -5);
\draw (1, 2, -4) -- (2, 2, -5);
\draw (2, 1, -4) -- (2, 2, -5);
\draw (2, 1, -4) -- (3, 1, -5);
\draw (2, 2, -3) -- (1, 2, -4);
\draw (2, 2, -3) -- (2, 1, -4);
\draw (2, 2, -3) -- (2, 3, -4);
\draw (2, 2, -3) -- (3, 2, -4);
\draw (2, 3, -4) -- (1, 3, -5);
\draw (2, 3, -4) -- (2, 2, -5);
\draw (3, 2, -4) -- (2, 2, -5);
\draw (3, 2, -4) -- (3, 1, -5);
\draw (3, 3, -3) -- (2, 3, -4);
\draw (3, 3, -3) -- (3, 2, -4);

\draw[gray, dashed] (1, 1, -5) -- (1, 3, -5) -- (3, 3, -5) -- (3, 1, -5) -- cycle;
\draw[gray, dashed] (1, 1, -3) -- (1, 3, -3) -- (3, 3, -3) -- (3, 1, -3) -- cycle;
\draw[gray, dashed] (1, 1, -5) -- (1, 1, -3);
\draw[gray, dashed] (1, 3, -5) -- (1, 3, -3);
\draw[gray, dashed] (3, 3, -5) -- (3, 3, -3);
\draw[gray, dashed] (3, 1, -5) -- (3, 1, -3);
\draw[gray, dashed] (1, 3, -5) -- (3, 1, -5);
\draw[gray, dashed] (1, 1, -3) -- (3, 3, -3);

\end{tikzpicture}
%
% Tetrahedron, n=5
\tdplotsetmaincoords{110}{130} 
\begin{tikzpicture}[tdplot_main_coords, scale=1.4]

\filldraw[black] (1, 1, -6) circle (1.5pt) node[anchor=west] {};
\filldraw[black] (1, 2, -7) circle (1.5pt) node[anchor=west] {};
\filldraw[black] (1, 3, -8) circle (1.5pt) node[anchor=west] {};
\filldraw[black] (1, 4, -9) circle (1.5pt) node[anchor=west] {};
\filldraw[black] (2, 1, -7) circle (1.5pt) node[anchor=west] {};
\filldraw[black] (2, 2, -8) circle (1.5pt) node[anchor=west] {};
\filldraw[black] (2, 2, -6) circle (1.5pt) node[anchor=west] {};
\filldraw[black] (2, 3, -9) circle (1.5pt) node[anchor=west] {};
\filldraw[black] (2, 3, -7) circle (1.5pt) node[anchor=west] {};
\filldraw[black] (2, 4, -8) circle (1.5pt) node[anchor=west] {};
\filldraw[black] (3, 1, -8) circle (1.5pt) node[anchor=west] {};
\filldraw[black] (3, 2, -9) circle (1.5pt) node[anchor=west] {};
\filldraw[black] (3, 2, -7) circle (1.5pt) node[anchor=west] {};
\filldraw[black] (3, 3, -8) circle (1.5pt) node[anchor=west] {};
\filldraw[black] (3, 3, -6) circle (1.5pt) node[anchor=west] {};
\filldraw[black] (3, 4, -7) circle (1.5pt) node[anchor=west] {};
\filldraw[black] (4, 1, -9) circle (1.5pt) node[anchor=west] {};
\filldraw[black] (4, 2, -8) circle (1.5pt) node[anchor=west] {};
\filldraw[black] (4, 3, -7) circle (1.5pt) node[anchor=west] {};
\filldraw[black] (4, 4, -6) circle (1.5pt) node[anchor=west] {};

\draw (1, 1, -6) -- (1, 2, -7);
\draw (1, 1, -6) -- (2, 1, -7);
\draw (1, 2, -7) -- (1, 3, -8);
\draw (1, 2, -7) -- (2, 2, -8);
\draw (1, 3, -8) -- (1, 4, -9);
\draw (1, 3, -8) -- (2, 3, -9);
\draw (2, 1, -7) -- (2, 2, -8);
\draw (2, 1, -7) -- (3, 1, -8);
\draw (2, 2, -8) -- (2, 3, -9);
\draw (2, 2, -8) -- (3, 2, -9);
\draw (2, 2, -6) -- (1, 2, -7);
\draw (2, 2, -6) -- (2, 1, -7);
\draw (2, 2, -6) -- (2, 3, -7);
\draw (2, 2, -6) -- (3, 2, -7);
\draw (2, 3, -7) -- (1, 3, -8);
\draw (2, 3, -7) -- (2, 2, -8);
\draw (2, 3, -7) -- (2, 4, -8);
\draw (2, 3, -7) -- (3, 3, -8);
\draw (2, 4, -8) -- (1, 4, -9);
\draw (2, 4, -8) -- (2, 3, -9);
\draw (3, 1, -8) -- (3, 2, -9);
\draw (3, 1, -8) -- (4, 1, -9);
\draw (3, 2, -7) -- (2, 2, -8);
\draw (3, 2, -7) -- (3, 1, -8);
\draw (3, 2, -7) -- (3, 3, -8);
\draw (3, 2, -7) -- (4, 2, -8);
\draw (3, 3, -8) -- (2, 3, -9);
\draw (3, 3, -8) -- (3, 2, -9);
\draw (3, 3, -6) -- (2, 3, -7);
\draw (3, 3, -6) -- (3, 2, -7);
\draw (3, 3, -6) -- (3, 4, -7);
\draw (3, 3, -6) -- (4, 3, -7);
\draw (3, 4, -7) -- (2, 4, -8);
\draw (3, 4, -7) -- (3, 3, -8);
\draw (4, 2, -8) -- (3, 2, -9);
\draw (4, 2, -8) -- (4, 1, -9);
\draw (4, 3, -7) -- (3, 3, -8);
\draw (4, 3, -7) -- (4, 2, -8);
\draw (4, 4, -6) -- (3, 4, -7);
\draw (4, 4, -6) -- (4, 3, -7);

\draw[gray, dashed] (1, 1, -9) -- (1, 4, -9) -- (4, 4, -9) -- (4, 1, -9) -- cycle;
\draw[gray, dashed] (1, 1, -6) -- (1, 4, -6) -- (4, 4, -6) -- (4, 1, -6) -- cycle;
\draw[gray, dashed] (1, 1, -9) -- (1, 1, -6);
\draw[gray, dashed] (1, 4, -9) -- (1, 4, -6);
\draw[gray, dashed] (4, 4, -9) -- (4, 4, -6);
\draw[gray, dashed] (4, 1, -9) -- (4, 1, -6);
\draw[gray, dashed] (1, 4, -9) -- (4, 1, -9);
\draw[gray, dashed] (1, 1, -6) -- (4, 4, -6);

\end{tikzpicture}
\caption{The tetrahedron for $n= 3,4,5$, respectively.}
\label{fig_tetra}
\end{figure}

Via the bijection given by Theorem~\ref{thm1}\eqref{thm1.2},
this gives the Hasse diagram for $(\mathbf{B}_n,\leq)$, cf. \cite[Figure 1 and 2]{EH}.

In Figure~\ref{fig_socles4} and Figure \ref{fig_socles5}, we give some examples of 
composition factors of the form $L_x$, where $x\in \mathcal{J}$,
in the modules $\Delta_e/\Delta_w$, for $n=4$ and $5$. 
The composition factors of this form in $\Delta_e/\Delta_w$ which 
are not in the socle are given by the black points in the figures  
and the composition factors in the socle are given by the 
red points in the figures.
The white points correspond to the composition factors in $\Delta_w$
(and hence outside of $\Delta_e/\Delta_w$). 
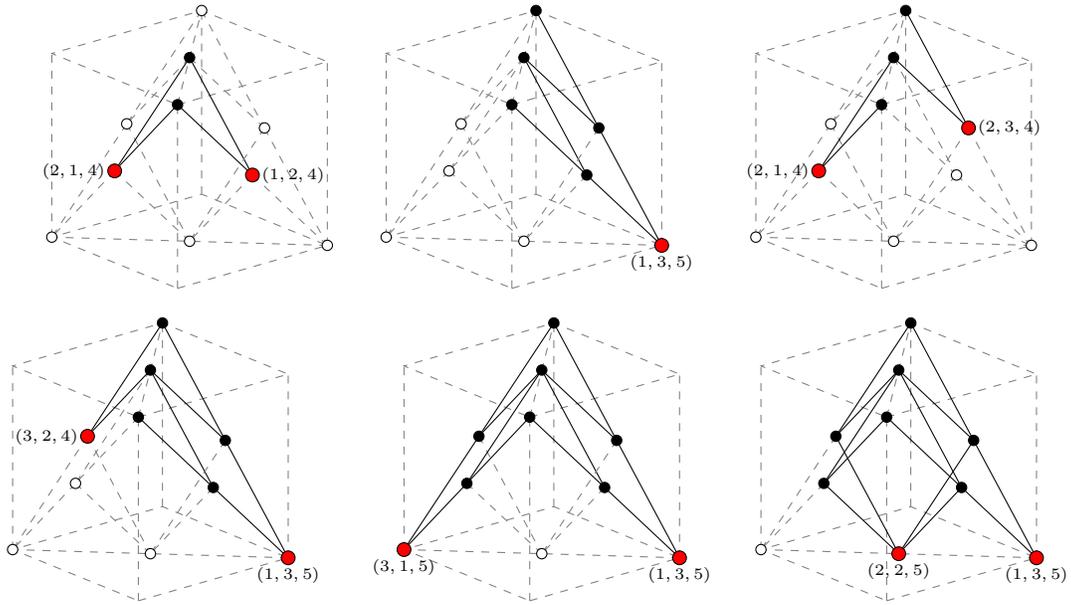
\begin{figure}[ht]
\centering
% coshadow(121), n=4
\tdplotsetmaincoords{110}{130} 
\begin{tikzpicture}[tdplot_main_coords, scale=1.3]

\draw[gray, dashed] (1, 1, -5) -- (1, 3, -5) -- (3, 3, -5) -- (3, 1, -5) -- cycle;
\draw[gray, dashed] (1, 1, -3) -- (1, 3, -3) -- (3, 3, -3) -- (3, 1, -3) -- cycle;
\draw[gray, dashed] (1, 1, -5) -- (1, 1, -3);
\draw[gray, dashed] (1, 3, -5) -- (1, 3, -3);
\draw[gray, dashed] (3, 3, -5) -- (3, 3, -3);
\draw[gray, dashed] (3, 1, -5) -- (3, 1, -3);
\draw[gray, dashed] (1, 3, -5) -- (3, 1, -5);
\draw[gray, dashed] (1, 1, -3) -- (3, 3, -3);

\draw (1, 1, -3) -- (2, 1, -4);
\draw (1, 1, -3) -- (1, 2, -4);
\draw[gray, dashed] (2, 1, -4) -- (3, 1, -5);
\draw[gray, dashed] (2, 1, -4) -- (2, 2, -5);
\draw[gray, dashed] (1, 2, -4) -- (2, 2, -5);
\draw[gray, dashed] (1, 2, -4) -- (1, 3, -5);
\draw[gray, dashed] (3, 2, -4) -- (3, 1, -5);
\draw[gray, dashed] (3, 2, -4) -- (2, 2, -5);
\draw (2, 2, -3) -- (2, 1, -4);
\draw (2, 2, -3) -- (1, 2, -4);
\draw[gray, dashed] (2, 2, -3) -- (3, 2, -4);
\draw[gray, dashed] (2, 2, -3) -- (2, 3, -4);
\draw[gray, dashed] (3, 3, -3) -- (3, 2, -4);
\draw[gray, dashed] (3, 3, -3) -- (2, 3, -4);
\draw[gray, dashed] (2, 3, -4) -- (2, 2, -5);
\draw[gray, dashed] (2, 3, -4) -- (1, 3, -5);

\filldraw[black] (1, 1, -3) circle (1.5pt) node[anchor=west] {};
\draw[fill=red] (2, 1, -4) circle (2.0pt) node[anchor=east] {\tiny $(2, 1, 4)$};
\draw[fill=white] (3, 1, -5) circle (1.5pt) node[anchor=west] {};
\draw[fill=red] (1, 2, -4) circle (2.0pt) node[anchor=west] {\tiny $(1, 2, 4)$};
\draw[fill=white] (3, 2, -4) circle (1.5pt) node[anchor=west] {};
\draw[fill=white] (2, 2, -5) circle (1.5pt) node[anchor=west] {};
\filldraw[black] (2, 2, -3) circle (1.5pt) node[anchor=west] {};
\draw[fill=white] (3, 3, -3) circle (1.5pt) node[anchor=west] {};
\draw[fill=white] (2, 3, -4) circle (1.5pt) node[anchor=west] {};
\draw[fill=white] (1, 3, -5) circle (1.5pt) node[anchor=west] {};

\end{tikzpicture}
% coshadow(123), n=4
\tdplotsetmaincoords{110}{130} 
\begin{tikzpicture}[tdplot_main_coords, scale=1.3]

\draw[gray, dashed] (1, 1, -5) -- (1, 3, -5) -- (3, 3, -5) -- (3, 1, -5) -- cycle;
\draw[gray, dashed] (1, 1, -3) -- (1, 3, -3) -- (3, 3, -3) -- (3, 1, -3) -- cycle;
\draw[gray, dashed] (1, 1, -5) -- (1, 1, -3);
\draw[gray, dashed] (1, 3, -5) -- (1, 3, -3);
\draw[gray, dashed] (3, 3, -5) -- (3, 3, -3);
\draw[gray, dashed] (3, 1, -5) -- (3, 1, -3);
\draw[gray, dashed] (1, 3, -5) -- (3, 1, -5);
\draw[gray, dashed] (1, 1, -3) -- (3, 3, -3);

\draw[gray, dashed] (1, 1, -3) -- (2, 1, -4);
\draw (1, 1, -3) -- (1, 2, -4);
\draw[gray, dashed] (2, 1, -4) -- (3, 1, -5);
\draw[gray, dashed] (2, 1, -4) -- (2, 2, -5);
\draw[gray, dashed] (1, 2, -4) -- (2, 2, -5);
\draw (1, 2, -4) -- (1, 3, -5);
\draw[gray, dashed] (3, 2, -4) -- (3, 1, -5);
\draw[gray, dashed] (3, 2, -4) -- (2, 2, -5);
\draw[gray, dashed] (2, 2, -3) -- (2, 1, -4);
\draw (2, 2, -3) -- (1, 2, -4);
\draw[gray, dashed] (2, 2, -3) -- (3, 2, -4);
\draw (2, 2, -3) -- (2, 3, -4);
\draw[gray, dashed] (3, 3, -3) -- (3, 2, -4);
\draw (3, 3, -3) -- (2, 3, -4);
\draw[gray, dashed] (2, 3, -4) -- (2, 2, -5);
\draw (2, 3, -4) -- (1, 3, -5);

\filldraw[black] (1, 1, -3) circle (1.5pt) node[anchor=west] {};
\draw[fill=white] (2, 1, -4) circle (1.5pt) node[anchor=west] {};
\draw[fill=white] (3, 1, -5) circle (1.5pt) node[anchor=west] {};
\filldraw[black] (1, 2, -4) circle (1.5pt) node[anchor=west] {};
\draw[fill=white] (3, 2, -4) circle (1.5pt) node[anchor=west] {};
\draw[fill=white] (2, 2, -5) circle (1.5pt) node[anchor=west] {};
\filldraw[black] (2, 2, -3) circle (1.5pt) node[anchor=west] {};
\filldraw[black] (3, 3, -3) circle (1.5pt) node[anchor=west] {};
\filldraw[black] (2, 3, -4) circle (1.5pt) node[anchor=west] {};
\draw[fill=red] (1, 3, -5) circle (2.0pt) node[anchor=north] {\tiny $(1, 3, 5)$};

\end{tikzpicture}
% coshadow(231), n=4
\tdplotsetmaincoords{110}{130} 
\begin{tikzpicture}[tdplot_main_coords, scale=1.3]

\draw[gray, dashed] (1, 1, -5) -- (1, 3, -5) -- (3, 3, -5) -- (3, 1, -5) -- cycle;
\draw[gray, dashed] (1, 1, -3) -- (1, 3, -3) -- (3, 3, -3) -- (3, 1, -3) -- cycle;
\draw[gray, dashed] (1, 1, -5) -- (1, 1, -3);
\draw[gray, dashed] (1, 3, -5) -- (1, 3, -3);
\draw[gray, dashed] (3, 3, -5) -- (3, 3, -3);
\draw[gray, dashed] (3, 1, -5) -- (3, 1, -3);
\draw[gray, dashed] (1, 3, -5) -- (3, 1, -5);
\draw[gray, dashed] (1, 1, -3) -- (3, 3, -3);

\draw (1, 1, -3) -- (2, 1, -4);
\draw[gray, dashed] (1, 1, -3) -- (1, 2, -4);
\draw[gray, dashed] (2, 1, -4) -- (3, 1, -5);
\draw[gray, dashed] (2, 1, -4) -- (2, 2, -5);
\draw[gray, dashed] (1, 2, -4) -- (2, 2, -5);
\draw[gray, dashed] (1, 2, -4) -- (1, 3, -5);
\draw[gray, dashed] (3, 2, -4) -- (3, 1, -5);
\draw[gray, dashed] (3, 2, -4) -- (2, 2, -5);
\draw (2, 2, -3) -- (2, 1, -4);
\draw[gray, dashed] (2, 2, -3) -- (1, 2, -4);
\draw[gray, dashed] (2, 2, -3) -- (3, 2, -4);
\draw (2, 2, -3) -- (2, 3, -4);
\draw[gray, dashed] (3, 3, -3) -- (3, 2, -4);
\draw (3, 3, -3) -- (2, 3, -4);
\draw[gray, dashed] (2, 3, -4) -- (2, 2, -5);
\draw[gray, dashed] (2, 3, -4) -- (1, 3, -5);

\filldraw[black] (1, 1, -3) circle (1.5pt) node[anchor=west] {};
\draw[fill=red] (2, 1, -4) circle (2.0pt) node[anchor=east] {\tiny $(2, 1, 4)$};
\draw[fill=white] (3, 1, -5) circle (1.5pt) node[anchor=west] {};
\draw[fill=white] (1, 2, -4) circle (1.5pt) node[anchor=west] {};
\draw[fill=white] (3, 2, -4) circle (1.5pt) node[anchor=west] {};
\draw[fill=white] (2, 2, -5) circle (1.5pt) node[anchor=west] {};
\filldraw[black] (2, 2, -3) circle (1.5pt) node[anchor=west] {};
\filldraw[black] (3, 3, -3) circle (1.5pt) node[anchor=west] {};
\draw[fill=red] (2, 3, -4) circle (2.0pt) node[anchor=west] {\tiny $(2, 3, 4)$};
\draw[fill=white] (1, 3, -5) circle (1.5pt) node[anchor=west] {};

\end{tikzpicture}
\vskip 0.3cm
% coshadow(1232), n=4
\tdplotsetmaincoords{110}{130} 
\begin{tikzpicture}[tdplot_main_coords, scale=1.3]

\draw[gray, dashed] (1, 1, -5) -- (1, 3, -5) -- (3, 3, -5) -- (3, 1, -5) -- cycle;
\draw[gray, dashed] (1, 1, -3) -- (1, 3, -3) -- (3, 3, -3) -- (3, 1, -3) -- cycle;
\draw[gray, dashed] (1, 1, -5) -- (1, 1, -3);
\draw[gray, dashed] (1, 3, -5) -- (1, 3, -3);
\draw[gray, dashed] (3, 3, -5) -- (3, 3, -3);
\draw[gray, dashed] (3, 1, -5) -- (3, 1, -3);
\draw[gray, dashed] (1, 3, -5) -- (3, 1, -5);
\draw[gray, dashed] (1, 1, -3) -- (3, 3, -3);

\draw[gray, dashed] (1, 1, -3) -- (2, 1, -4);
\draw (1, 1, -3) -- (1, 2, -4);
\draw[gray, dashed] (2, 1, -4) -- (3, 1, -5);
\draw[gray, dashed] (2, 1, -4) -- (2, 2, -5);
\draw[gray, dashed] (1, 2, -4) -- (2, 2, -5);
\draw (1, 2, -4) -- (1, 3, -5);
\draw[gray, dashed] (3, 2, -4) -- (3, 1, -5);
\draw[gray, dashed] (3, 2, -4) -- (2, 2, -5);
\draw[gray, dashed] (2, 2, -3) -- (2, 1, -4);
\draw (2, 2, -3) -- (1, 2, -4);
\draw (2, 2, -3) -- (3, 2, -4);
\draw (2, 2, -3) -- (2, 3, -4);
\draw (3, 3, -3) -- (3, 2, -4);
\draw (3, 3, -3) -- (2, 3, -4);
\draw[gray, dashed] (2, 3, -4) -- (2, 2, -5);
\draw (2, 3, -4) -- (1, 3, -5);

\filldraw[black] (1, 1, -3) circle (1.5pt) node[anchor=west] {};
\draw[fill=white] (2, 1, -4) circle (1.5pt) node[anchor=west] {};
\draw[fill=white] (3, 1, -5) circle (1.5pt) node[anchor=west] {};
\filldraw[black] (1, 2, -4) circle (1.5pt) node[anchor=west] {};
\draw[fill=red] (3, 2, -4) circle (2.0pt) node[anchor=east] {\tiny $(3, 2, 4)$};
\draw[fill=white] (2, 2, -5) circle (1.5pt) node[anchor=west] {};
\filldraw[black] (2, 2, -3) circle (1.5pt) node[anchor=west] {};
\filldraw[black] (3, 3, -3) circle (1.5pt) node[anchor=west] {};
\filldraw[black] (2, 3, -4) circle (1.5pt) node[anchor=west] {};
\draw[fill=red] (1, 3, -5) circle (2.0pt) node[anchor=north] {\tiny $(1, 3, 5)$};

\end{tikzpicture}
% coshadow(12321), n=4
\tdplotsetmaincoords{110}{130} 
\begin{tikzpicture}[tdplot_main_coords, scale=1.3]

\draw[gray, dashed] (1, 1, -5) -- (1, 3, -5) -- (3, 3, -5) -- (3, 1, -5) -- cycle;
\draw[gray, dashed] (1, 1, -3) -- (1, 3, -3) -- (3, 3, -3) -- (3, 1, -3) -- cycle;
\draw[gray, dashed] (1, 1, -5) -- (1, 1, -3);
\draw[gray, dashed] (1, 3, -5) -- (1, 3, -3);
\draw[gray, dashed] (3, 3, -5) -- (3, 3, -3);
\draw[gray, dashed] (3, 1, -5) -- (3, 1, -3);
\draw[gray, dashed] (1, 3, -5) -- (3, 1, -5);
\draw[gray, dashed] (1, 1, -3) -- (3, 3, -3);

\draw (1, 1, -3) -- (2, 1, -4);
\draw (1, 1, -3) -- (1, 2, -4);
\draw (2, 1, -4) -- (3, 1, -5);
\draw[gray, dashed] (2, 1, -4) -- (2, 2, -5);
\draw[gray, dashed] (1, 2, -4) -- (2, 2, -5);
\draw (1, 2, -4) -- (1, 3, -5);
\draw (3, 2, -4) -- (3, 1, -5);
\draw[gray, dashed] (3, 2, -4) -- (2, 2, -5);
\draw (2, 2, -3) -- (2, 1, -4);
\draw (2, 2, -3) -- (1, 2, -4);
\draw (2, 2, -3) -- (3, 2, -4);
\draw (2, 2, -3) -- (2, 3, -4);
\draw (3, 3, -3) -- (3, 2, -4);
\draw (3, 3, -3) -- (2, 3, -4);
\draw[gray, dashed] (2, 3, -4) -- (2, 2, -5);
\draw (2, 3, -4) -- (1, 3, -5);

\filldraw[black] (1, 1, -3) circle (1.5pt) node[anchor=west] {};
\filldraw[black] (2, 1, -4) circle (1.5pt) node[anchor=west] {};
\draw[fill=red] (3, 1, -5) circle (2.0pt) node[anchor=north] {\tiny $(3, 1, 5)$};
\filldraw[black] (1, 2, -4) circle (1.5pt) node[anchor=west] {};
\filldraw[black] (3, 2, -4) circle (1.5pt) node[anchor=west] {};
\draw[fill=white] (2, 2, -5) circle (1.5pt) node[anchor=west] {};
\filldraw[black] (2, 2, -3) circle (1.5pt) node[anchor=west] {};
\filldraw[black] (3, 3, -3) circle (1.5pt) node[anchor=west] {};
\filldraw[black] (2, 3, -4) circle (1.5pt) node[anchor=west] {};
\draw[fill=red] (1, 3, -5) circle (2.0pt) node[anchor=north] {\tiny $(1, 3, 5)$};

\end{tikzpicture}
% coshadow(12312), n=4
\tdplotsetmaincoords{110}{130} 
\begin{tikzpicture}[tdplot_main_coords, scale=1.3]

\draw[gray, dashed] (1, 1, -5) -- (1, 3, -5) -- (3, 3, -5) -- (3, 1, -5) -- cycle;
\draw[gray, dashed] (1, 1, -3) -- (1, 3, -3) -- (3, 3, -3) -- (3, 1, -3) -- cycle;
\draw[gray, dashed] (1, 1, -5) -- (1, 1, -3);
\draw[gray, dashed] (1, 3, -5) -- (1, 3, -3);
\draw[gray, dashed] (3, 3, -5) -- (3, 3, -3);
\draw[gray, dashed] (3, 1, -5) -- (3, 1, -3);
\draw[gray, dashed] (1, 3, -5) -- (3, 1, -5);
\draw[gray, dashed] (1, 1, -3) -- (3, 3, -3);

\draw (1, 1, -3) -- (2, 1, -4);
\draw (1, 1, -3) -- (1, 2, -4);
\draw[gray, dashed] (2, 1, -4) -- (3, 1, -5);
\draw (2, 1, -4) -- (2, 2, -5);
\draw (1, 2, -4) -- (2, 2, -5);
\draw (1, 2, -4) -- (1, 3, -5);
\draw[gray, dashed] (3, 2, -4) -- (3, 1, -5);
\draw (3, 2, -4) -- (2, 2, -5);
\draw (2, 2, -3) -- (2, 1, -4);
\draw (2, 2, -3) -- (1, 2, -4);
\draw (2, 2, -3) -- (3, 2, -4);
\draw (2, 2, -3) -- (2, 3, -4);
\draw (3, 3, -3) -- (3, 2, -4);
\draw (3, 3, -3) -- (2, 3, -4);
\draw (2, 3, -4) -- (2, 2, -5);
\draw (2, 3, -4) -- (1, 3, -5);

\filldraw[black] (1, 1, -3) circle (1.5pt) node[anchor=west] {};
\filldraw[black] (2, 1, -4) circle (1.5pt) node[anchor=west] {};
\draw[fill=white] (3, 1, -5) circle (1.5pt) node[anchor=west] {};
\filldraw[black] (1, 2, -4) circle (1.5pt) node[anchor=west] {};
\filldraw[black] (3, 2, -4) circle (1.5pt) node[anchor=west] {};
\draw[fill=red] (2, 2, -5) circle (2.0pt) node[anchor=north] {\tiny $(2, 2, 5)$};
\filldraw[black] (2, 2, -3) circle (1.5pt) node[anchor=west] {};
\filldraw[black] (3, 3, -3) circle (1.5pt) node[anchor=west] {};
\filldraw[black] (2, 3, -4) circle (1.5pt) node[anchor=west] {};
\draw[fill=red] (1, 3, -5) circle (2.0pt) node[anchor=north] {\tiny $(1, 3, 5)$};

\end{tikzpicture}
\caption{Socles of the quotients $\Delta_e/\Delta_w$ for $n=4$, for $w=s_1s_2s_1, s_1s_2s_3, s_2s_3s_1$, respectively, in the first row, and $w=s_1s_2s_3s_2, s_1s_2s_3s_2s_1, s_1s_2s_3s_1s_2$, respectively, in the second row.}
\label{fig_socles4}
\end{figure}

\begin{figure}[ht]
\centering
% coshadow(3412321), n=5
\tdplotsetmaincoords{110}{130} 
\begin{tikzpicture}[tdplot_main_coords, scale=1.3]

\draw[gray, dashed] (1, 1, -9) -- (1, 4, -9) -- (4, 4, -9) -- (4, 1, -9) -- cycle;
\draw[gray, dashed] (1, 1, -6) -- (1, 4, -6) -- (4, 4, -6) -- (4, 1, -6) -- cycle;
\draw[gray, dashed] (1, 1, -9) -- (1, 1, -6);
\draw[gray, dashed] (1, 4, -9) -- (1, 4, -6);
\draw[gray, dashed] (4, 4, -9) -- (4, 4, -6);
\draw[gray, dashed] (4, 1, -9) -- (4, 1, -6);
\draw[gray, dashed] (1, 4, -9) -- (4, 1, -9);
\draw[gray, dashed] (1, 1, -6) -- (4, 4, -6);

\draw (4, 4, -6) -- (4, 3, -7);
\draw (4, 4, -6) -- (3, 4, -7);
\draw (4, 3, -7) -- (4, 2, -8);
\draw (4, 3, -7) -- (3, 3, -8);
\draw (4, 2, -8) -- (4, 1, -9);
\draw[gray, dashed] (4, 2, -8) -- (3, 2, -9);
\draw[gray, dashed] (3, 3, -8) -- (3, 2, -9);
\draw[gray, dashed] (3, 3, -8) -- (2, 3, -9);
\draw (3, 3, -6) -- (4, 3, -7);
\draw (3, 3, -6) -- (3, 4, -7);
\draw (3, 3, -6) -- (3, 2, -7);
\draw (3, 3, -6) -- (2, 3, -7);
\draw (3, 4, -7) -- (3, 3, -8);
\draw[gray, dashed] (3, 4, -7) -- (2, 4, -8);
\draw (3, 1, -8) -- (4, 1, -9);
\draw[gray, dashed] (3, 1, -8) -- (3, 2, -9);
\draw (3, 2, -7) -- (4, 2, -8);
\draw (3, 2, -7) -- (3, 3, -8);
\draw (3, 2, -7) -- (3, 1, -8);
\draw[gray, dashed] (3, 2, -7) -- (2, 2, -8);
\draw (1, 1, -6) -- (1, 2, -7);
\draw (1, 1, -6) -- (2, 1, -7);
\draw (1, 2, -7) -- (1, 3, -8);
\draw[gray, dashed] (1, 2, -7) -- (2, 2, -8);
\draw[gray, dashed] (1, 3, -8) -- (1, 4, -9);
\draw[gray, dashed] (1, 3, -8) -- (2, 3, -9);
\draw (2, 1, -7) -- (3, 1, -8);
\draw[gray, dashed] (2, 1, -7) -- (2, 2, -8);
\draw (2, 3, -7) -- (3, 3, -8);
\draw (2, 3, -7) -- (1, 3, -8);
\draw[gray, dashed] (2, 3, -7) -- (2, 2, -8);
\draw[gray, dashed] (2, 3, -7) -- (2, 4, -8);
\draw[gray, dashed] (2, 2, -8) -- (3, 2, -9);
\draw[gray, dashed] (2, 2, -8) -- (2, 3, -9);
\draw (2, 2, -6) -- (3, 2, -7);
\draw (2, 2, -6) -- (1, 2, -7);
\draw (2, 2, -6) -- (2, 1, -7);
\draw (2, 2, -6) -- (2, 3, -7);
\draw[gray, dashed] (2, 4, -8) -- (1, 4, -9);
\draw[gray, dashed] (2, 4, -8) -- (2, 3, -9);

\filldraw[black] (4, 4, -6) circle (1.5pt) node[anchor=north] {};
\filldraw[black] (4, 3, -7) circle (1.5pt) node[anchor=north] {};
\draw[fill=red] (4, 1, -9) circle (2.0pt) node[anchor=north] {\tiny $(4, 1, 9)$};
\filldraw[black] (4, 2, -8) circle (1.5pt) node[anchor=north] {};
\draw[fill=red] (3, 3, -8) circle (2.0pt) node[anchor=north] {\tiny $(3, 3, 8)$};
\filldraw[black] (3, 3, -6) circle (1.5pt) node[anchor=north] {};
\filldraw[black] (3, 4, -7) circle (1.5pt) node[anchor=north] {};
\filldraw[black] (3, 1, -8) circle (1.5pt) node[anchor=north] {};
\draw[fill=white] (3, 2, -9) circle (1.5pt) node[anchor=north] {};
\filldraw[black] (3, 2, -7) circle (1.5pt) node[anchor=north] {};
\filldraw[black] (1, 1, -6) circle (1.5pt) node[anchor=north] {};
\filldraw[black] (1, 2, -7) circle (1.5pt) node[anchor=north] {};
\draw[fill=white] (1, 4, -9) circle (1.5pt) node[anchor=north] {};
\draw[fill=red] (1, 3, -8) circle (2.0pt) node[anchor=north] {\tiny $(1, 3, 8)$};
\filldraw[black] (2, 1, -7) circle (1.5pt) node[anchor=north] {};
\draw[fill=white] (2, 3, -9) circle (1.5pt) node[anchor=north] {};
\filldraw[black] (2, 3, -7) circle (1.5pt) node[anchor=north] {};
\draw[fill=white] (2, 2, -8) circle (1.5pt) node[anchor=north] {};
\filldraw[black] (2, 2, -6) circle (1.5pt) node[anchor=north] {};
\draw[fill=white] (2, 4, -8) circle (1.5pt) node[anchor=north] {};

\end{tikzpicture}
\hskip 1cm
% coshadow(12342312), n=5
\tdplotsetmaincoords{110}{130} 
\begin{tikzpicture}[tdplot_main_coords, scale=1.3]

\draw[gray, dashed] (1, 1, -9) -- (1, 4, -9) -- (4, 4, -9) -- (4, 1, -9) -- cycle;
\draw[gray, dashed] (1, 1, -6) -- (1, 4, -6) -- (4, 4, -6) -- (4, 1, -6) -- cycle;
\draw[gray, dashed] (1, 1, -9) -- (1, 1, -6);
\draw[gray, dashed] (1, 4, -9) -- (1, 4, -6);
\draw[gray, dashed] (4, 4, -9) -- (4, 4, -6);
\draw[gray, dashed] (4, 1, -9) -- (4, 1, -6);
\draw[gray, dashed] (1, 4, -9) -- (4, 1, -9);
\draw[gray, dashed] (1, 1, -6) -- (4, 4, -6);

\draw (4, 4, -6) -- (4, 3, -7);
\draw (4, 4, -6) -- (3, 4, -7);
\draw (4, 3, -7) -- (4, 2, -8);
\draw (4, 3, -7) -- (3, 3, -8);
\draw[gray, dashed] (4, 2, -8) -- (4, 1, -9);
\draw (4, 2, -8) -- (3, 2, -9);
\draw (3, 3, -8) -- (3, 2, -9);
\draw[gray, dashed] (3, 3, -8) -- (2, 3, -9);
\draw (3, 3, -6) -- (4, 3, -7);
\draw (3, 3, -6) -- (3, 4, -7);
\draw (3, 3, -6) -- (3, 2, -7);
\draw (3, 3, -6) -- (2, 3, -7);
\draw (3, 4, -7) -- (3, 3, -8);
\draw (3, 4, -7) -- (2, 4, -8);
\draw[gray, dashed] (3, 1, -8) -- (4, 1, -9);
\draw (3, 1, -8) -- (3, 2, -9);
\draw (3, 2, -7) -- (4, 2, -8);
\draw (3, 2, -7) -- (3, 3, -8);
\draw (3, 2, -7) -- (3, 1, -8);
\draw (3, 2, -7) -- (2, 2, -8);
\draw (1, 1, -6) -- (1, 2, -7);
\draw (1, 1, -6) -- (2, 1, -7);
\draw (1, 2, -7) -- (1, 3, -8);
\draw (1, 2, -7) -- (2, 2, -8);
\draw (1, 3, -8) -- (1, 4, -9);
\draw[gray, dashed] (1, 3, -8) -- (2, 3, -9);
\draw (2, 1, -7) -- (3, 1, -8);
\draw (2, 1, -7) -- (2, 2, -8);
\draw (2, 3, -7) -- (3, 3, -8);
\draw (2, 3, -7) -- (1, 3, -8);
\draw (2, 3, -7) -- (2, 2, -8);
\draw (2, 3, -7) -- (2, 4, -8);
\draw (2, 2, -8) -- (3, 2, -9);
\draw[gray, dashed] (2, 2, -8) -- (2, 3, -9);
\draw (2, 2, -6) -- (3, 2, -7);
\draw (2, 2, -6) -- (1, 2, -7);
\draw (2, 2, -6) -- (2, 1, -7);
\draw (2, 2, -6) -- (2, 3, -7);
\draw (2, 4, -8) -- (1, 4, -9);
\draw[gray, dashed] (2, 4, -8) -- (2, 3, -9);

\filldraw[black] (4, 4, -6) circle (1.5pt) node[anchor=north] {};
\filldraw[black] (4, 3, -7) circle (1.5pt) node[anchor=north] {};
\draw[fill=white] (4, 1, -9) circle (1.5pt) node[anchor=north] {};
\filldraw[black] (4, 2, -8) circle (1.5pt) node[anchor=north] {};
\filldraw[black] (3, 3, -8) circle (1.5pt) node[anchor=north] {};
\filldraw[black] (3, 3, -6) circle (1.5pt) node[anchor=north] {};
\filldraw[black] (3, 4, -7) circle (1.5pt) node[anchor=north] {};
\filldraw[black] (3, 1, -8) circle (1.5pt) node[anchor=north] {};
\draw[fill=red] (3, 2, -9) circle (2.0pt) node[anchor=north] {\tiny $(3, 2, 9)$};
\filldraw[black] (3, 2, -7) circle (1.5pt) node[anchor=north] {};
\filldraw[black] (1, 1, -6) circle (1.5pt) node[anchor=north] {};
\filldraw[black] (1, 2, -7) circle (1.5pt) node[anchor=north] {};
\draw[fill=red] (1, 4, -9) circle (2.0pt) node[anchor=north] {\tiny $(1, 4, 9)$};
\filldraw[black] (1, 3, -8) circle (1.5pt) node[anchor=north] {};
\filldraw[black] (2, 1, -7) circle (1.5pt) node[anchor=north] {};
\draw[fill=white] (2, 3, -9) circle (1.5pt) node[anchor=north] {};
\filldraw[black] (2, 3, -7) circle (1.5pt) node[anchor=north] {};
\filldraw[black] (2, 2, -8) circle (1.5pt) node[anchor=north] {};
\filldraw[black] (2, 2, -6) circle (1.5pt) node[anchor=north] {};
\filldraw[black] (2, 4, -8) circle (1.5pt) node[anchor=north] {};

\end{tikzpicture}
\caption{Socles of the quotients $\Delta_e/\Delta_w$ for $n=5$, for $w=s_3s_4s_1s_2s_3s_2s_1$ and $w=s_1s_2s_3s_4s_2s_3s_1s_2$, respectively.}
\label{fig_socles5}
\end{figure}
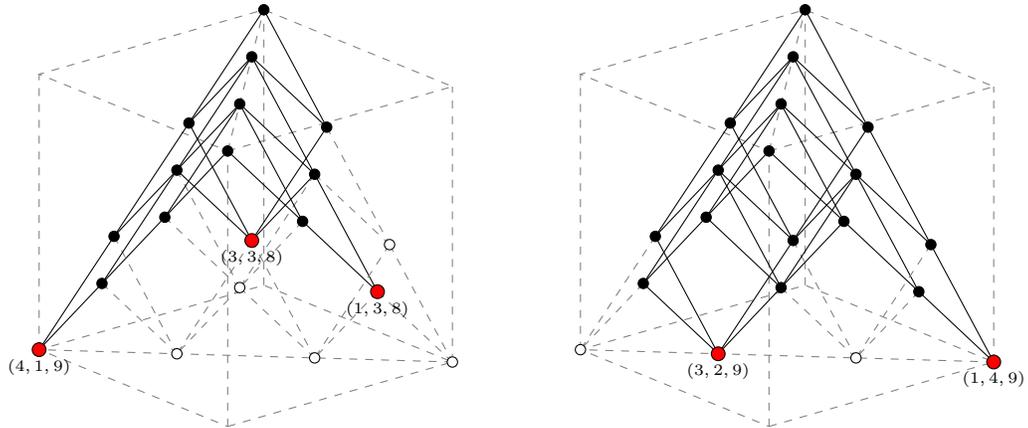

For convenience of the reader, on Figure \ref{fig_J4} and Figure \ref{fig_J5} we present the penultimate cell, for $n=4$ and $n=5$, respectively. In the figures, an element on the position $(i,j)$ has $i$ as the unique left ascend, and $j$ as the unique right ascend.

\begin{figure}[ht]
\centering
\begin{tabular}{c|ccc}
 & \small{$1$} & \small{$2$} & \small{$3$}\\ \hline
\small{$1$} & 
{$s_2s_3s_2$} & {$s_2s_3s_2s_1$} & {$s_2s_3s_1s_2s_1$}\\
\small{$2$} & 
{$s_1s_2s_3s_2$} & {$s_1s_2s_3s_2s_1$} & {$s_3s_1s_2s_1$}\\
\small{$3$} & 
{$s_1s_2s_3s_1s_2$} & {$s_1s_2s_3s_1$} & {$s_1s_2s_1$}\\
\end{tabular}
\caption{The penultimate cell for $n=4$.}
\label{fig_J4}

\bigskip

\begin{tabular}{c|cccc}
 & \small{$1$} & \small{$2$} & \small{$3$} & \small{$4$}\\ \hline
\small{$1$} & 
{$s_2s_3s_4s_2s_3s_2$} & {$s_2s_3s_4s_2s_3s_2s_1$} & {$s_2s_3s_4s_2s_3s_1s_2s_1$} & {$s_2s_3s_4s_1s_2s_3s_1s_2s_1$}\\
\small{$2$} & 
{$s_1s_2s_3s_4s_2s_3s_2$} & {$s_1s_2s_3s_4s_2s_3s_2s_1$} & {$s_1s_2s_3s_4s_2s_3s_1s_2s_1$} & {$s_3s_4s_1s_2s_3s_1s_2s_1$}\\
\small{$3$} & 
{$s_1s_2s_3s_4s_1s_2s_3s_2$} & {$s_1s_2s_3s_4s_1s_2s_3s_2s_1$} & {$s_1s_2s_3s_4s_3s_1s_2s_1$} & {$s_4s_1s_2s_3s_1s_2s_1$}\\
\small{$4$} & 
{$s_1s_2s_3s_4s_1s_2s_3s_1s_2$} & {$s_1s_2s_3s_4s_1s_2s_3s_1$} & {$s_1s_2s_3s_4s_1s_2s_1$} & {$s_1s_2s_3s_1s_2s_1$}\\
\end{tabular}
\caption{The penultimate cell for $n=5$.}
\label{fig_J5}
\end{figure}

\subsection{Explicit formulae}\label{subs_explicit}\label{s3.2}

The set $\mathbf{B}_n$ can be explicitly parameterized by triples of integers 
\begin{displaymath}
(i,j,k)\quad \text{ such that } \quad1 \leq i, j \leq n-1\quad
\text{ and }\quad0 \leq k \leq \min\{i-1,j-1,n-1-i,n-1-j\}.
\end{displaymath}
For $i \leq j$, put:
\begin{displaymath}
    b(i,j,k) := (s_is_{i+1} \ldots s_{j+k})(s_{i-1}s_i \ldots s_{j+k-1}) \cdots
    (s_{i-k}s_{i-k+1} \ldots s_j),
\end{displaymath}
and if $j<i$, set $b(i,j,k) := b(j,i,k)^{-1}$.

We can draw $w\in S_n$ as a picture on a two-dimensional grid, in the following way: View $w\in S_n$ as a permutation on $\{1,\cdots, n\}$ acting on the left, and put
\begin{tikzpicture}[scale=0.6]
\draw[fill=white] (0,0) circle (3.3pt);
\end{tikzpicture} 
on positions $(i,w(i))$, $i=1,\ldots, n$, in matrix coordinates. This way, we can visualize bigrassmannian permutations as on Figure \ref{fig_bigrass}.

\begin{figure}[ht]
\centering
\begin{tikzpicture}[scale=0.6, yscale=-1]
% Axes:
\draw (0.5,-0.25)--(0.5,16.2);
\draw (-0.25,0.5)--(16.2,0.5);
\node[anchor=east] at (0.5,1) {\scriptsize{$1$}};
\node[] at (1,0) {\scriptsize{$1$}};
\node[anchor=east]at (0.5,2) {\scriptsize{$2$}};
\node[] at (2,0) {\scriptsize{$2$}};
\node[anchor=east] at (0.5,3) {$\vdots$};
\node[] at (3,0) {$\ldots$};
\node[anchor=east] at (0.5,4) {\scriptsize{$r$}};
\node[] at (4,0) {\scriptsize{$r$}};
\node[anchor=east] at (0.5,5) {\scriptsize{$r+1$}};
\node[] at (5.2,0) {\scriptsize{$r+1$}};
\node[anchor=east] at (0.5,6) {\scriptsize{$r+2$}};
%\node[] at (6,0) {\scriptsize{$r+2$}};
\node[anchor=east] at (0.5,7) {$\vdots$};
\node[] at (7,0) {$\ldots$};
\node[anchor=east] at (0.5,8) {\scriptsize{$j$}};
\node[] at (8,0) {\scriptsize{$i$}};
\node[anchor=east] at (0.5,9) {\scriptsize{$j+1$}};
\node[] at (9.1,0) {\scriptsize{$i+1$}};
\node[anchor=east] at (0.5,10) {\scriptsize{$j+2$}};
%\node[] at (10,0) {\scriptsize{$10$}};
\node[anchor=east] at (0.5,11) {$\vdots$};
\node[] at (10.1,0) {$\ldots$};
\node[anchor=east] at (0.5,12) {\scriptsize{$i+j-r$}};
\node[] at (11.5,0) {\scriptsize{$i+j-r$}};
\node[anchor=east] at (0.5,13) {\scriptsize{$i+j-r+1$}};
%\node[] at (13.7,0){\scriptsize{$i+j-r+1$}};
\node[anchor=east] at (0.5,14) {\scriptsize{$i+j-r+2$}};
%\node[] at (14,0) {\scriptsize{$14$}};
\node[anchor=east] at (0.5,15) {$\vdots$};
\node[] at (15.2,0) {$\ldots$};
\node[anchor=east] at (0.5,16) {\scriptsize{$n$}};
\node[] at (16,0) {\scriptsize{$n$}};
%  Points:
\draw[fill=white] (1,1) circle (3.3pt);
\draw[fill=white] (2,2) circle (3.3pt);
\node[] at (3,3) {$\ddots$};
\draw[fill=white] (4,4) circle (3.3pt);
\draw[fill=white] (9,5) circle (3.3pt);
\draw[fill=white] (10,6) circle (3.3pt);
\node[] at (11,7) {$\ddots$};
\draw[fill=white] (12,8) circle (3.3pt);
\draw[fill=white] (5,9) circle (3.3pt);
\draw[fill=white] (6,10) circle (3.3pt);
\node[] at (7,11) {$\ddots$};
\draw[fill=white] (8,12) circle (3.3pt);
\draw[fill=white] (13,13) circle (3.3pt);
\draw[fill=white] (14,14) circle (3.3pt);
\node[] at (15,15) {$\ddots$};
\draw[fill=white] (16,16) circle (3.3pt);
%  Dashed:
\draw[dashed] (4.5,-0.25)--(4.5,16.25);
\draw[dashed] (-0.25,4.5)--(16.25,4.5);
\draw[dashed] (8.5,-0.25)--(8.5,16.25);
\draw[dashed] (-0.25,8.5)--(16.25,8.5);
\draw[dashed] (12.5,-0.25)--(12.5,16.25);
\draw[dashed] (-0.25,12.5)--(16.25,12.5);
\end{tikzpicture}
\caption{The bigrassmannian $b(i,j,k)$. Here $r=  \min\{i-1,j-1\}-k$.}
\label{fig_bigrass}
\end{figure}
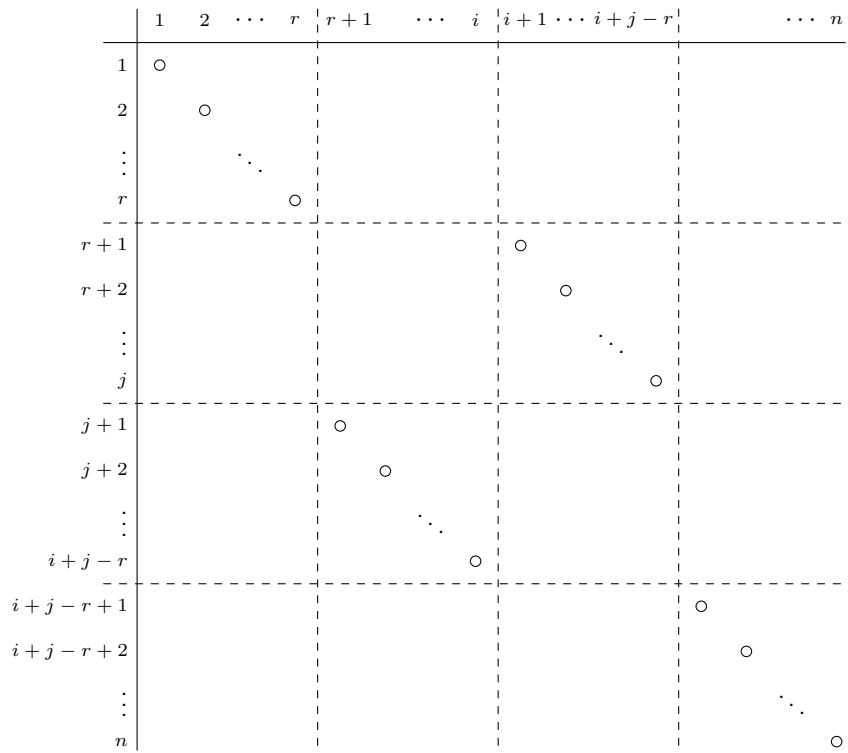

\begin{proposition}
\label{prop_explicit}
The bijection from Theorem~\ref{thm1}(\ref{thm1.2}) is given by
\[ b(i,j,k) \mapsto L_{w_{i,j}} \left\langle -\frac{(n-1)(n-2)}{2}- |i-j|-2k  \right\rangle \]
\end{proposition}
\begin{proof}
The assignment in the claim does define a bijection by Proposition~\ref{prop7}, so we need to check that it agrees with the bijection in Theorem~\ref{thm1}.
From the definition, we see that $b(i,j,k)<b(i,j,k')$ (in the Bruhat order), for $k<k'$. This property, since the relevant subquotients appear multiplicity-freely, determines the bijection.
Since the inclusion of the Verma modules $\Delta_x$ agree with (the opposite of) the Bruhat order, the bijection in Theorem~\ref{thm1} also has this property and thus agrees with the one given in our claim.
\end{proof}

\subsection{Essential set and rank of a permutation}\label{s3.3} 

For $w \in S_n$, the so-called {\em essential set} attached to $w$ (defined in \cite{Fu}) is given by
\[ \Ess(w) := \{ (i,j) \colon 1 \leq i,j \leq n-1, \ i<w^{-1}(j), \ j<w(i), \ w(i+1) \leq j, \ w^{-1}(j+1) \leq i  \}. \]

In Figure \ref{fig_ess5} we give examples of essential sets for $w \in S_5$. To easily find $\Ess(w)$, following \cite{Ko1}, we put 
\begin{tikzpicture}[scale=0.7]
\draw[fill=white] (0,0) circle (3pt);
\end{tikzpicture} 
on positions $(i,w(i))$, $i=1,\ldots, n$, in matrix coordinates, 
and kill all cells to the right or below of these. The surviving cells are denoted by
\begin{tikzpicture}[scale=0.7]
\draw (0,0) node[cross] {};
\end{tikzpicture},
and are sometimes called the {\em diagram} of $w$. The south-east corners of the diagram (denoted by 
\begin{tikzpicture}[scale=0.7]
\draw (0,0) node[cross] {};
\node at (0,0) [rectangle,draw] {};
\end{tikzpicture}) 
constitute the essential set attached to $w$.

\begin{figure}[ht]
\centering
% Ess(3412321)
\begin{tikzpicture}[scale=0.7, yscale=-1]

% Axes:
\draw (0.5,-0.25)--(0.5,5.25);
\draw (-0.25,0.5)--(5.25,0.5);
\node[] at (0,1) {\small{$1$}};
\node[] at (1,0) {\small{$1$}};
\node[] at (0,2) {\small{$2$}};
\node[] at (2,0) {\small{$2$}};
\node[] at (0,3) {\small{$3$}};
\node[] at (3,0) {\small{$3$}};
\node[] at (0,4) {\small{$4$}};
\node[] at (4,0) {\small{$4$}};
\node[] at (0,5) {\small{$5$}};
\node[] at (5,0) {\small{$5$}};

%  Points:
\draw (1,1) node[cross] {};
\draw (2,1) node[cross] {};
\draw (3,1) node[cross] {};
\draw (4,1) node[cross] {};
\node at (4,1) [rectangle,draw] {};
\draw (5,1)--(5,5);
\draw (5,1)--(5,1);
\draw[fill=white] (5,1) circle (3pt);
\draw (1,2) node[cross] {};
\draw (2,2)--(2,5);
\draw (2,2)--(5,2);
\draw[fill=white] (2,2) circle (3pt);
\draw (1,3) node[cross] {};
\node at (1,3) [rectangle,draw] {};
\draw (3,3) node[cross] {};
\node at (3,3) [rectangle,draw] {};
\draw (4,3)--(4,5);
\draw (4,3)--(5,3);
\draw[fill=white] (4,3) circle (3pt);
\draw (1,4)--(1,5);
\draw (1,4)--(5,4);
\draw[fill=white] (1,4) circle (3pt);
\draw (3,5)--(3,5);
\draw (3,5)--(5,5);
\draw[fill=white] (3,5) circle (3pt);
\end{tikzpicture}
\hskip 1.5cm
% Ess(12342312)
\begin{tikzpicture}[scale=0.7, yscale=-1]

% Axes:
\draw (0.5,-0.25)--(0.5,5.25);
\draw (-0.25,0.5)--(5.25,0.5);
\node[] at (0,1) {\small{$1$}};
\node[] at (1,0) {\small{$1$}};
\node[] at (0,2) {\small{$2$}};
\node[] at (2,0) {\small{$2$}};
\node[] at (0,3) {\small{$3$}};
\node[] at (3,0) {\small{$3$}};
\node[] at (0,4) {\small{$4$}};
\node[] at (4,0) {\small{$4$}};
\node[] at (0,5) {\small{$5$}};
\node[] at (5,0) {\small{$5$}};

%  Points:
\draw (1,1) node[cross] {};
\draw (2,1) node[cross] {};
\draw (3,1) node[cross] {};
\draw (4,1)--(4,5);
\draw (4,1)--(5,1);
\draw[fill=white] (4,1) circle (3pt);
\draw (1,2) node[cross] {};
\draw (2,2) node[cross] {};
\draw (3,2) node[cross] {};
\node at (3,2) [rectangle,draw] {};
\draw (5,2)--(5,5);
\draw (5,2)--(5,2);
\draw[fill=white] (5,2) circle (3pt);
\draw (1,3) node[cross] {};
\draw (2,3)--(2,5);
\draw (2,3)--(5,3);
\draw[fill=white] (2,3) circle (3pt);
\draw (1,4) node[cross] {};
\node at (1,4) [rectangle,draw] {};
\draw (3,4)--(3,5);
\draw (3,4)--(5,4);
\draw[fill=white] (3,4) circle (3pt);
\draw (1,5)--(1,5);
\draw (1,5)--(5,5);
\draw[fill=white] (1,5) circle (3pt);
\end{tikzpicture}
\caption{Essential sets attached to $s_3s_4s_1s_2s_3s_2s_1$ and $s_1s_2s_3s_4s_2s_3s_1s_2$ (for $n=5$), respectively.}
\label{fig_ess5}
\end{figure}
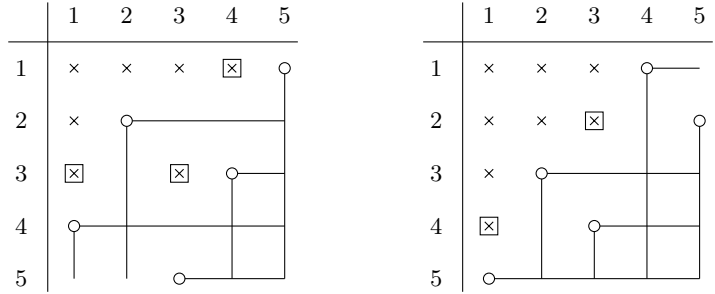

In \cite{Ko1} it is shown that there is a bijection between
$\mathbf{BM}(w)$ and $\Ess(w)$. More precisely, an element $(i,j) \in \Ess(w)$ corresponds to a certain monotone triangle (see the cited article for definition), denoted by $J_{i,k,j+1}$, for some $k$, which is then identified with a bigrassmannian permutation. From the description of monotone triangles in \cite[Section 8]{Re}, it follows that $J_{i,k,j+1}$ correspond to a bigrassmannian permutation with left descent $j$ and right descent $i$. From Proposition \ref{prop9} we know that such a bigrassmannian element 
in $\mathbf{BM}(w)$ is uniquely determined. We conclude:

\begin{corollary}[A reformulation of {\cite[Theorem]{Ko1}}]
\label{cor_Ess}
For $w \in S_n$, the map \[x \mapsto (\text{right descent of }x, \text{left descent of }x)\] is a bijection $\mathbf{BM}(w) \to \Ess(w)$.
\end{corollary}

The following result allows us to determine the socle of 
$\Delta_e/\Delta_w$ via the essential set of $w$ which is very easy and efficient to compute.

\begin{corollary}
\label{cor_ess}
The (ungraded) socle of $\Delta_e/\Delta_w$ is given by 
$\displaystyle\bigoplus_{(i,j) \in \Ess(w)} L_{w_{j,i}}$. 
\end{corollary}
\begin{proof}
The claim follows from Proposition \ref{prop_explicit} and Corollary \ref{cor_Ess}.
\end{proof}

To illustrate how this corollary works, one can compare 
Figure~\ref{fig_ess5} with Figure~\ref{fig_socles5} using Figure~\ref{fig_J5}.
Observe that the essential set alone does not provide 
information on the degrees of the composition factors 
in the socle of $\Delta_e/\Delta_w$. To get these degrees,
we need another combinatorial tool called the {\em rank function}.

The rest of the subsection provides an upgrade of the description in 
Corollary~\ref{cor_ess} to the graded setup.
For $w \in S_n$, the so-called \emph{rank function} $r_w$ (defined in \cite{Fu}) is given by:
\[ r_w(i,j) := |\{k \leq i \colon w(k) \leq j  \}|, \qquad 1 \leq i,j \leq n.  \]

More useful for us is the function $t_w$, which we call the {\em co-rank function}, given by
\[t_w(i,j) := \min\{i,j\}-r_w(i,j). \]
If we again consider a permutation as a picture on a two-dimensional grid as before, then $r_w(i,j)$ is equal to the number of \begin{tikzpicture}[scale=0.7] \draw[fill=white] (0,0) circle (3pt); \end{tikzpicture}  in the north-west area in Figure \ref{fig_cork_gen}.
If $i \leq j$, then $t_w(i,j)$ is equal to the number of \begin{tikzpicture}[scale=0.7] \draw[fill=white] (0,0) circle (3pt); \end{tikzpicture}  in the north-east area in Figure \ref{fig_cork_gen}, and otherwise to the number of \begin{tikzpicture}[scale=0.7] \draw[fill=white] (0,0) circle (3pt); \end{tikzpicture}  in the south-west area.

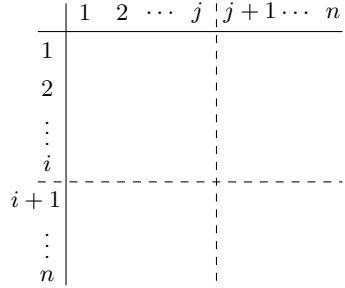
\begin{figure}[ht]
\centering
\begin{tikzpicture}[scale=0.5, yscale=-1]
% Axes:
\draw (0.5,-0.25)--(0.5,7.25);
\draw (-0.25,0.5)--(7.85,0.5);
\node[] at (0,1) {\small{$1$}};
\node[] at (1,0) {\small{$1$}};
\node[] at (0,2) {\small{$2$}};
\node[] at (2,0) {\small{$2$}};
\node[] at (0,3) {\small{$\vdots$}};
\node[] at (3,0) {\small{$\ldots$}};
\node[] at (0,4) {\small{$i$}};
\node[] at (4,0) {\small{$j$}};
\node[] at (-.3,5) {\small{$i+1$}};
\node[] at (5.4,0) {\small{$j+1$}};
\node[] at (0,6) {\small{$\vdots$}};
\node[] at (6.6,0) {\small{$\ldots$}};
\node[] at (0,7) {\small{$n$}};
\node[] at (7.6,0) {\small{$n$}};

\draw[dashed] (4.5,-0.25)--(4.5,7.25);
\draw[dashed] (-0.25,4.5)--(7.85,4.5);
%\node[] at (4,4) {\small{$(i,j)$}};
\end{tikzpicture}
\caption{Areas which determine $r_w(i,j)$ and $t_w(i,j)$.}
\label{fig_cork_gen}
\end{figure}

The co-rank functions for the examples in Figure \ref{fig_ess5} are given in Figure \ref{fig_corank5}. 

\begin{figure}[ht]
\centering
\begin{tabular}{c|ccccc}
 & \small{$1$} & \small{$2$} & \small{$3$} & \small{$4$} & \small{$5$}\\ \hline
\small{$1$} & 
\small{$1$} & \small{$1$} & \small{$1$} & \fbox{\small{$1$}} & \small{$0$}\\
\small{$2$} & 
\small{$1$} & \small{$1$} & \small{$1$} & \small{$1$} & \small{$0$}\\
\small{$3$} & 
\fbox{\small{$1$}} & \small{$1$} & \fbox{\small{$2$}} & \small{$1$} & \small{$0$}\\
\small{$4$} & 
\small{$0$} & \small{$0$} & \small{$1$} & \small{$1$} & \small{$0$}\\
\small{$5$} & 
\small{$0$} & \small{$0$} & \small{$0$} & \small{$0$} & \small{$0$}\\
\end{tabular}
\hskip 1.5cm
\begin{tabular}{c|ccccc}
 & \small{$1$} & \small{$2$} & \small{$3$} & \small{$4$} & \small{$5$}\\ \hline
\small{$1$} & 
\small{$1$} & \small{$1$} & \small{$1$} & \small{$0$} & \small{$0$}\\
\small{$2$} & 
\small{$1$} & \small{$2$} & \fbox{\small{$2$}} & \small{$1$} & \small{$0$}\\
\small{$3$} & 
\small{$1$} & \small{$1$} & \small{$2$} & \small{$1$} & \small{$0$}\\
\small{$4$} & 
\fbox{\small{$1$}} & \small{$1$} & \small{$1$} & \small{$1$} & \small{$0$}\\
\small{$5$} & 
\small{$0$} & \small{$0$} & \small{$0$} & \small{$0$} & \small{$0$}\\
\end{tabular}
\caption{The values of $t_w(i,j)$ for  $w=s_3s_4s_1s_2s_3s_2s_1$ and $w=s_1s_2s_3s_4s_2s_3s_1s_2$ (for $n=5$), respectively. The values at $\Ess(w)$ are boxed.}
\label{fig_corank5}
\end{figure}

\begin{lemma}\label{t=bruhat}
For $w,x\in S_n$, we have $t_w\leq t_x$ if and only if $w\leq x$.
\end{lemma}
\begin{proof}
See \cite[Theorem 2.1.5]{BB2}.
\end{proof}

The following lemma is clear from Figure \ref{fig_bigrass}:

\begin{lemma}\label{tbig}
For each $b(i,j,k)\in \mathbf{B}_n$, we have $\Ess(b(i,j,k))=\{(j,i)\}$ and $t_{b(i,j,k)}(j,i)=k+1$.
\end{lemma}

Finally, we can describe the graded shifts of the socle constituents in
$\Delta_e/(\Delta_w\langle- \ell(w)\rangle)$.

\begin{proposition}\label{grsoc}
Let $w\in S_n$. The socle of $N_w:=\Delta_e/(\Delta_w\langle- \ell(w)\rangle)$ in $^{\mathbb{Z}} \cO$ is given by 
\[%\bigoplus_{(i,j)\in\Ess(w)} L_{w_{j,i}}\langle\mathbf{a}(\mathcal J)+|i-j|+2(t_w(i,j)-1) \rangle =
\bigoplus_{(i,j)\in\Ess(w)} L_{w_{j,i}}\left\langle-\frac{(n-1)(n-2)}{2}-|i-j|-2(t_w(i,j)-1) \right\rangle.\]
\end{proposition}

\begin{proof}
Note first that, for $w=b(j,i,k)\in\mathbf{B}_n$, we have
\begin{equation}\label{socbig}
\soc(N_w) =L_{w_{j,i}}\langle 
-\mathbf{a}(\mathcal{J})-|i-j|-2(t_w(i,j)-1) \rangle 
%L_{w_{i,j}}\left\langle \frac{(n-1)(n-2)}{2}-|i-j|+2k \right\rangle 
\end{equation}
by Lemma~\ref{tbig} and Proposition~\ref{prop_explicit}.
Recall that $\mathbf{a}(\mathcal{J})=\frac{(n-1)(n-2)}{2}$.

Now let $w\in S_n$ be arbitrary and let $(i,j)\in \Ess(w)$. 
By Corollary \ref{cor_ess} we only need to determine the degree shift, say $m$, of $L_{w_{j,i}}$ in the socle of $N_w$.
Let $b(j,i,k)$ be the element in $\mathbf{BM}(w)$ which corresponds to $(i,j)$ under the bijection in Corollary \ref{cor_Ess}. 
Since $b(j,i,k)\in\mathbf{BM}(w)$, by Lemma \ref{t=bruhat} we have $b(j,i,k)\leq w$ while $ b(j,i,k+1)\not\leq w$, and thus $N_w\not\surj N_{b(j,i,k+1)}$ while $N_w \surj N_{b(j,i,k)}$. So $m = \mathbf{a}(\mathcal{J})+|i-j|+2(t_w(i,j)-1)$ follows from \eqref{socbig} and 
the parity of $m$.
\end{proof}

\section{Further remarks}\label{srem}

\subsection{Inclusions between arbitrary Verma modules}\label{s7.1}

An immediate consequence of Theorem~\ref{thm1} is:

\begin{corollary}\label{cor21}
Let $v,w\in S_n$ be such that $v<w$.
\begin{enumerate}[$($i$)$]
\item\label{cor21.1} The bijection from Theorem~\ref{thm1}\eqref{thm1.2} induces
a bijection between simple subquotients of 
$\Delta_v/\Delta_w$ of the form $L_x$, where $x\in\mathcal{J}$, 
and $y\in \mathbf{B}_n$ such that $y\leq w$ and $y\not\leq v$.
\item\label{cor21.2} The socle of $\Delta_v/\Delta_w$ 
consists of all $L_x$, where $x$ corresponds to an element in $\mathbf{BM}(w)\setminus \mathbf{BM}(v)$.
\end{enumerate}
\end{corollary}

A more detailed description of the socle of
$\Delta_v\langle - \ell(v)\rangle/(\Delta_w\langle - \ell(w)\rangle)$ 
as an object in  $^\mathbb Z\cO_0$ follows from Proposition~\ref{grsoc}.

\subsection{No such clean result in other types}\label{s7.3}

Unfortunately, Theorem~\ref{thm1} is not true, in general, in other types. 
One of the reasons is that 
Corollary~\ref{cor8} fails already in types $B_3$ and $D_4$.
We note that $\mathbf B_n$ agrees in type A with the set of join-irreducible elements in $W$, while, in general, there are bigrassmannian elements that are not join-irreducible. 
To generalize Theorem \ref{thm1}, we need, to start with, replace $\mathbf{B}_n$ by the set of join-irreducible elements in $W$. 
But even then, most of our crucial arguments fail outside type A.
For example, in non-simply laced types, for some pairs of simple reflections
$s$ and $t$ there will be more than one element $w\in\mathcal{J}$
such that $sw>w$ and $wt>w$. Another problem is that neither bigrassmannian elements nor join-irreducible elements
 with fixed left and right descents form a chain
with respect to the Bruhat order. 

Rank two case is, however,  special. In this case $\mathcal{J}$ is the set of
bigrassmannian elements and all KL polynomials are trivial. So,
Theorem~\ref{thm1} is true. Notably, an appropriate analogue of
the map $\Phi$ in this case is not the identity map.

\subsection*{Acknowledgments}

This research was partially supported by
the Swedish Research Council, 
G{\"o}ran Gustafsson Stiftelse and Vergstiftelsen. 
The third author was also partially supported by the 
QuantiXLie Center of Excellence grant no. KK.01.1.1.01.0004 
funded by the European Regional Development Fund.

We are especially indebted to Sascha Orlik
and Matthias Strauch whose question
started this research.

\vspace{2mm}

\noindent
H.~K.: Department of Mathematics, Uppsala University, Box. 480,
SE-75106, Uppsala,\\ SWEDEN, email: {\tt hankyung.ko\symbol{64}math.uu.se}

\noindent
V.~M.: Department of Mathematics, Uppsala University, Box. 480,
SE-75106, Uppsala,\\ SWEDEN, email: {\tt mazor\symbol{64}math.uu.se}

\noindent
R.~M.: Department of Mathematics, Uppsala University, Box. 480,
SE-75106, Uppsala,\\ SWEDEN, email: {\tt rafael.mrden\symbol{64}math.uu.se} \\
(On leave from: Faculty of Civil Engineering, University of Zagreb, \\
Fra Andrije Ka\v{c}i\'{c}a-Mio\v{s}i\'{c}a 26, 10000 Zagreb, CROATIA)

\end{document}